%% file: main.tex
\documentclass[11pt,letterpaper]{article}
\usepackage{amsfonts, amsmath, amssymb, amscd, amsthm, color, graphicx, mathrsfs, wasysym, setspace, mdwlist, calc, float, mathtools, tikz-cd, calc}

\usepackage{multirow, makecell}
\usepackage{comment}
\usepackage{array, caption}
\newcolumntype{?}{!{\vrule width 1pt}}

\usepackage{colortbl}
\usepackage{hyperref}
\usepackage{tocloft}
\definecolor{DarkBlue}{RGB}{20,20,200}
\definecolor{DarkGreen}{RGB}{20,120,20}

\usepackage{xcolor}

\usepackage{hyperref}
\hypersetup{linktocpage}

\hypersetup{colorlinks,
    linkcolor={red!50!black},
    citecolor={blue!80!black},
    urlcolor={blue!80!black}}
\usepackage{float}

\renewcommand{\kappa}{\varkappa}
\renewcommand{\phi}{\varphi}
\newcommand{\e}{\varepsilon}

\newcommand{\NN}{\mathbb{N}}
\newcommand{\NNo}{\mathbb{N}\cup\{0\}}

\newcommand{\RR}{\mathbb{R}}

\renewcommand{\SS}{\mathcal{S}}

\renewcommand{\d}{{\rm d}}

\newcommand{\diam}{\operatorname{diam}}
\renewcommand{\H}{\mathcal{H}}

\newcommand{\R}{\mathfrak{R}}
\newcommand{\K}{\mathcal{K}}
\newcommand{\I}{\mathcal{I}}
\newcommand{\rank}{\mathfrak{r}}

\newcommand{\h}{\mathfrak{h}}
\renewcommand{\S}{\mathfrak{S}}

\newcommand{\act}{\curvearrowright}

\newcommand{\supp}{\mathrm{supp}}
\newcommand{\hy}{\text{-}}

\newcommand{\pred}{\mathcal{H}_s(G)_{\mathrm{pre}}}
\newcommand{\pmax}{\mathcal{H}_s(G)_{\mathrm{max}}}

\newcommand{\T}{\mathcal{T}}
\newcommand{\V}{\mathcal{V}}
\newcommand{\st}{\mathrm{St}}

\newtheorem{thm}{Theorem}[section]
\newtheorem*{thm*}{Theorem}

\newtheorem{lem}[thm]{Lemma}
\newtheorem{prop}[thm]{Proposition}

\theoremstyle{definition}
\newtheorem{defn}[thm]{Definition}

\theoremstyle{remark}
\newtheorem{rem}[thm]{Remark}

\newfont{\eufm}{eufm10}

\begin{document}
\title{\vspace*{-5mm}Borel asymptotic dimension of the Roller boundary of finite dimensional CAT(0) cube complexes}
\author{Koichi Oyakawa}
\date{}

\maketitle
\vspace*{-4mm}

\begin{abstract}
We prove that for any countable finite dimensional CAT(0) cube complex, the Borel median graph on its Roller compactification has the Borel asymptotic dimension bounded from above by its dimension.
\end{abstract}

\section{Introduction}

Borel complexity measures how complicated it is to describe equivalence relations on standard Borel spaces. This complexity is measured by comparing two Borel equivalence relations by Borel reducibility and by establishing land mark properties of Borel equivalence relations. Smoothness, hyperfiniteness, and treeability are among these landmark properties, the latter being more complicated.

One of actively studied Borel equivalence relations is the one defined by connected components of a Borel graph. The Borel graph is a collection of connected graphs as a set but admits a nice measurable structure, which allows us to ask definability of various graph theoretic and geometric notions. In \cite{CPTT23}, Chen-Poulin-Tao-Tserunyan proved several results showing that every locally finite Borel graph whose large-scale geometry is ``tree-like" induces a treeable equivalence relation. One of their crucial results, which derived the other results, was that every Borel median graphing of a countable Borel equivalence relation (CBER) with finite hyperplanes (i.e. the length of every strip of squares is finite) has a Borel subtreeing. They also proved that the CBER induced by a Borel median graph with stronger conditions is hyperfinite. The median graph is another name for (the 1-skeleton of) a CAT(0) cube complex, which was intensively studied in geometry. It can be considered either as a higher-dimensional generalization of a tree or as a tree-like generalization of a grid. Therefore, Borel median graphs sit a priori in a higher Borel complexity than Borel trees.

The purpose of this paper is to study the Borel complexity of Borel median graphs that appear naturally in geometry, that is, the Roller compactification of a CAT(0) cube complex. It turns out that this Borel complexity is much simpler than treeability. More precisely, we prove the following theorem.
\begin{thm}\label{thm:main}
    For any countable CAT(0) cube complex $X$ of dimension $n \in \NNo$, the Borel median graph on the Roller boundary $\R(X)$ of $X$ has Borel asymptotic dimension at most $n$ and its Borel equivalence relation consisting of connected components is smooth.
\end{thm}
Note that we don't assume local finiteness of a graph nor finiteness of hyperplanes (in the sense of \cite{CPTT23}). Borel asymptotic dimension was recently introduced by Conley, Jackson, Marks, Seward, and Tucker-Drob in \cite{CJM23}. It can be considered as a stronger notion than hyperfiniteness and played an important role in their progress toward Weiss's conjecture. It also inspired a new direction of research, being established as a new landmark in the hierarchy of Borel complexity (see \cite{BY24,GH24,IS24,Wei24}).

Readers have to be careful that although it is known by \cite{Wri12} that finite dimensional CAT(0) cube complexes have finite asymptotic dimension, the Borel asymptotic dimension of a Borel graph cannot be predicted from the asymptotic dimension of its connected components in general. Indeed, the free part of the orbit equivalence relation of the Bernoulli shift of the free group $F_2$ on $2^{F_2}$ (see \cite[Theorem 7.4.10]{Gao09}) is treeable, hence the asymptotic dimension of every equivalence class is at most 1. However, this Borel tree has infinite Borel asymptotic dimension by \cite[Theorem 1.7]{CJM23} since it is not hyperfinite. Moreover, by \cite[Corollary 8.2]{CJM23}, there even exist hyperfinite Borel trees with infinite Borel asymptotic dimension (see \cite[Section 8]{CJM23}).

Due to this distinction between Borel and usual asymptotic dimensions, Theorem \ref{thm:main} is worthwhile. Theorem \ref{thm:main} also strengthens \cite{Wri12}, because the asymptotic dimension of every connected component of a Borel graph is bounded from above by the Borel asymptotic dimension of the Borel graph and the original CAT(0) cube complex is a connected component (in the Borel graph) of its Roller compactification.

In Theorem \ref{thm:main}, it is straightforward to prove smoothness and most of this paper is devoted to prove finite Borel asymptotic dimension. In fact, we prove the following more general result (see Definition \ref{def:graph} for notations).
\begin{prop}\label{prop:intro smooth median graph}
    Let $X$ be a standard Borel space and $G \subset X^2$ be a countable Borel median graph such that $E_G^X$ is smooth. Suppose that there exists $D \in \NNo$ such that every $G$-component is a CAT(0) cube complex of dimension at most $D$. Then, $\mathrm{asdim}_{\mathbf{B}}(X,G) \le D$.
\end{prop}
There are three points to note on Proposition \ref{prop:intro smooth median graph}. First, although both smoothness and finite Borel asymptotic dimension are stronger properties than hyperfiniteness, neither implies the other in general (see Remark \ref{rem:smooth and finite Borel asy dim}). Secondly, the smoothness condition in Proposition \ref{prop:intro smooth median graph} is the best possible, because it is impossible to weaken to the case when $E_G^X$ is hyperfinite because of the existence of the hyperfinite Borel trees with infinite Borel asymptotic dimension as mentioned above. Finally, \cite[Theorem 4.2]{CJM23} is not applicable to prove Proposition \ref{prop:intro smooth median graph}, because we don't assume local finiteness of Borel graphs. This point is important to prove Theorem \ref{thm:main}, because even if we restrict to the case of locally finite CAT(0) cube complexes, their Roller compactification still can be locally infinite (see Remark \ref{rem:locally infinite Roller boundary}).

The strategy to prove Proposition \ref{prop:intro smooth median graph} is a Borel version of Wright's argument in \cite{Wri12}. For this, we introduce a Borel extended metric space that we obtain by gluing cubes to a Borel median graph. That is, given a Borel median graph, we define a natural Borel structure on the union of CAT(0) cube complexes corresponding to connected components of the Borel median graph. This idea allows us to handle the metric of Borel median graphs in a more sophisticated way and is a new insight of this paper, which might be useful for future study of Borel median graphs.

This paper is organized as follows. In Section \ref{sec:Preliminary}, we discuss the necessary definitions and known results about CAT(0) cube complexes and descriptive set theory. In Section \ref{sec:Gluing cubes to a Borel median graph}, we introduce the idea of gluing cubes to a Borel median graph. In Section \ref{sec:Standard Borel space of hyperplanes}, we study smooth Borel median graphs and introduce a quotient of a Borel median graph. In Section \ref{sec: Adding cubes to a Borel median graph (smooth version)}, we continue the study in Section \ref{sec:Standard Borel space of hyperplanes} and introduce a certain Borel extended metric space to which we embed a Borel median graph. In Section \ref{sec:Borel measurability of Wright's construction}, we introduce a Borel version of the Wright's construction. In Section \ref{sec:Proof of the main theorem}, we prove Proposition \ref{prop:intro smooth median graph} and then Theorem \ref{thm:main}.

\noindent\textbf{Acknowledgment.}
I would like to thank Talia Fernós for questions that inspired this work.

%Related work on Borel asymptotic dimension to check (?)
%(Victor Chepoi, Mark Hagen) Coloring hyperplanes of CAT(0) cube complexes (relevant?)

\section{Preliminary}\label{sec:Preliminary}

\subsection{CAT(0) cube complex}\label{sec:CAT(0) cube complex}

In this section, we recall various notions related to CAT(0) cube complexes. Some of them are presented in a different way from the literature, but they are essentially the same and tailored to be applicable to Borel setting. For more on CAT(0) cube complexes, readers are referred to \cite{Sag95,Rol98,Wis12,Sag14,Wis21}, although the list is too far from exhaustive.

\begin{defn}
    For $n \in \NNo$, a $n$-\emph{cube} is $[0,1]^n$. (A $0$-cube is a singleton by definition.) A \emph{cube complex} is a CW complex such that each cell is a $n$–cube for some $n$ and each cell is attached using an isometry of some face. Given a cube complex $X$ and $n\in \NNo$, we denote by $X^{(n)}$ the set of all $n$-cubes of $X$ and define $n$-\emph{skeleton of} $X$ to be the subcompex formed by $\bigcup_{i=0}^n X^{(i)}$. The \emph{dimension} of a cube complex $X$ is the least $N \in \NNo\cup\{\infty\}$ such that $X = \bigcup_{i=0}^N X^{(i)}$. We say $X$ is \emph{finite dimensional} if $N<\infty$.
\end{defn}

\begin{rem}
    For a graph $X$, which is a 1-dimensional cube complex, $X^{(0)}$ is the set of vertices and $X^{(1)}$ is the set of unoriented edges.
\end{rem}

\begin{defn}\label{def:median graph}
    A \emph{median graph} is a connected graph $X$ such that for any vertices $x_1,x_2,x_3\in X^{(0)}$, there exists a unique vertex $y\in X^{(0)}$ satisfying
    \begin{align*}
        d_X(x_i,x_j) = d_X(x_i,y)+d_X(y,x_j)
    \end{align*}
    for any distinct $i,j \in \{1,2,3\}$, where $d_X$ is the graph metric on $X$. The vertex $y$ is called the \emph{median of} $x_1,x_2,x_3$. A subset $B\subset X^{(0)}$ is called \emph{convex} if for any $x,y \in B$ and any geodesic $\gamma$ in $X$ from $x$ to $y$, we have $\gamma^{(0)} \subset B$. A \emph{CAT(0) cube complex} is a cube complex obtained from a median graph $X$ by gluing cubes to $X$ as follows; first whenever $X$ has an induced square, attach to it a 2-cube isometrically along edges, then next whenever there is a subcomplex isomorphic to the 2-skeleton of a 3-cube, attach to it a 3-cube isometrically along faces, and continue this procedure inductively for all dimensional cubes.
\end{defn}

\begin{rem}
    A CAT(0) cube complex is often defined as a simply connected cube complex such that the link of every vertex is a flag simplicial complex. This definition is equivalent to Definition \ref{def:median graph} by \cite{Che00}. In other words, a median graph is the 1-skeleton of a CAT(0) cube complex.
\end{rem}

Hyperplanes of a CAT(0) cube complex is usually defined using mid-cubes. However, we take a different approach to define hyperplanes, which is essentially the same as that of mid-cubes but seems more compatible with Borel setting (see Section \ref{sec:Standard Borel space of hyperplanes}).

\begin{defn}\label{def:hyperplane}
    Let $X$ be a CAT(0) cube complex. For $e,f \in X^{(1)}$, we define $e\sim f$ if either $e=f$ holds or there exist $e=e_0, e_1,\cdots,e_n=f \in X^{(1)}$ such that $e_{i-1}$ and $e_i$ are opposite edges of a square in $X$ for any $i \in  \{1,\cdots,n\}$. We define a \emph{hyperplane} to be an equivalence class of the relation $\sim$ in $X^{(1)}$. We denote by $\H(X)$ the set of all hyperplanes of $X$ (i.e. $\H(X) = X^{(1)}/\sim$). Given $h \in \H(X)$, the graph obtained by removing the interior of all edges belonging to $h$ from the 1-skeleton of $X$ has exactly two connected components. We define \emph{halfspaces delimited by} $h$ to be the 0-skeletons of these connected components, and denote them by $h^-,h^+ \subset X^{(0)}$ (note $X^{(0)} = h^- \sqcup h^+$). We say that $h^-$ is a \emph{choice of orientation for} $h$. For $h \in \H(X)$, we define $h^{(0)} \subset X^{(0)}$ by $h^{(0)} = \bigcup_{e \in h}e^{(0)}$, where $e^{(0)}$ is the set of the two endpoints of an edge $e$. Given $A,B \subset X^{(0)}$, we say that $h \in \H(X)$ \emph{separates $A$ and $B$} if $(A \subset h^- \,\wedge\,B \subset h^+) \,\vee\, (B \subset h^- \,\wedge\,A \subset h^+)$.
\end{defn}

\begin{rem}\label{rem:separate convex sets}
    Let $X$ be a CAT(0) cube complex. For any $h \in \H(X)$, $h^{(0)}$, $h^-$, and $h^+$ are convex. For any $x \in X^{(0)}$ and any convex subset $B \subset X^{(0)}$, there exists a unique vertex $y \in B$ such that $d_X(x,y) = \min_{y'\in B} d_X(x,y')$. Also, $d_X(x,B)$ is equal to the number of hyperplanes that separate $x$ and $B$.
\end{rem}

\begin{defn}\label{def:Roller boundary}
    Let $X$ be a CAT(0) cube complex. Given a vertex $v \in X^{(0)}$, for every $h\in \H(X)$, there exists exactly one halfspace $\h_v$ such that $v \in \h_v$. This defines an element $\alpha_v \in \prod_{h \in \H(X)}\{h^-,h^+\}$ by $\alpha_v(h)=\h_v$ for any $h \in \H(X)$. The \emph{Roller compactification $\R(X)$ of} $X$ is defined by $\R(X) = \overline{\{\alpha_v \mid v \in X^{(0)}\}}$, where each $\{h^-,h^+\}$ admits the discrete topology of two elements and $\prod_{h \in \H(X)}\{h^-,h^+\}$ admits the product topology. Define a graph $G_{\R(X)}\subset \R(X)^2$ by (see Definition \ref{def:graph})
    \begin{align*}
        G_{\R(X)} = \big\{(\alpha,\beta) \in \R(X)^2 \mid \#\{h \in \H(X)\mid \alpha(h) \neq \beta(h)\}=1 \big\}.
    \end{align*}
\end{defn}

\begin{rem}\label{rem:Roller boundary}
    It is well known and not difficult to see that the map $X^{(0)}\ni v \mapsto \alpha_v \in \R(X)$ is injective and we have $\R(X) = \big\{\alpha\in \prod_{h \in \H(X)}\{h^-,h^+\} \mid \forall\,h,k \in \H(X),\, \alpha
    (h)\cap \alpha(k) \neq \emptyset\big\}$. Also, $\alpha,\beta \in \R(X)$ are in the same connected component of the graph $G_{\R(X)}$ if and only if $\#\{h \in \H(X)\mid \alpha(h) \neq \beta(h)\}<\infty$. Every connected component of $G_{\R(X)}$ is a median graph. Indeed, for any $\alpha_1,\alpha_2,\alpha_3 \in \R(X)$ in the same connected component of $G_{\R(X)}$, their median $\beta$ is given by the condition $\forall\, h \in \H(X),\,\#\big\{i \in \{1,2,3\} \mid \beta(h)=\alpha_i(h)\big\}\ge 2$.
\end{rem}

\begin{rem}\label{rem:locally infinite Roller boundary}
    We can construct a locally finite 2-dimensional CAT(0) cube complex $X$ such that $\R(X)$ has a vertex of infinite valency as follows. Let $X_0$ be an infinite-length strip of squares formed by the vertices $\{(k,0),(k,1) \mid k\in\NNo\}$ in $\RR^2$. For every $n \in \NN$, we prepare a new CAT(0) cube complex $X_n$ isomorphic to $X_0$ and attach $X_n$ to $X_0$ by identifying the line $\{(x,0)\mid x\ge 0\}$ in $X_n$ and the line $\{(x+n,0)\mid x\ge 0\}$ in $X_0$. Let $X$ be the resulting CAT(0) cube complex. The limit point in $\R(X)$ of the sequence $\{(k,0)\}_{k\in\NN}$ belonging to $X_0 \subset X$ has infinite valency.
\end{rem}

Finally, we recall the application of the Sageev-Roller duality to the quotient of CAT(0) cube complexes. Unlike conventions in many literature, we do not assume that a CAT(0) cube complex is finite dimensional.

\begin{defn}\label{def:Sageev-Roller duality}
    Let $X$ be a CAT(0) cube complex and let $\K\subset \H(X)$. Define an equivalence relation $E_\K^X$ on $X^{(0)}$ as follows; for $v,w \in X^{(0)}$, $v \,E_\K^X\, w$ if and only if $\alpha_v(h) = \alpha_w(h)$ for any $h \in \K$. Denote the quotient $X^{(0)}/E_\K^X$ by $X_\K$ and the equivalence class containing $v \in X^{(0)}$ by $[v]$. Define a graph $G_\K \subset (X_\K)^2$ by
    \begin{align*}
        G_\K = \big\{([v],[w]) \in (X_\K)^2 \mid \#\{h \in \K\mid \alpha_v(h) \neq \alpha_w(h)\}=1 \big\},
    \end{align*}
    then $(X_\K,G_\K)$ becomes a median graph.
\end{defn}

\begin{rem}\label{rem:quotient of CAT(0) cc}
    Note $d_{X_\K}([v],[w]) = \#\{h \in \K\mid \alpha_v(h) \neq \alpha_w(h)\}$ for any $v,w\in X^{(0)}$. Also, the quotient map $X^{(0)} \to X_\K$ induces a bijection $\K \to \H(X_\K)$ by sending an edge $(v,w)\in (X^{(0)})^2$ belonging to a hyperplane in $\K$ to the edge $([v],[w]) \in (X_\K)^2$, which preserves transversality and inclusion of halfspaces. Hence, we will often identify $\K$ and $\H(X_\K)$. By this identification, when we take quotients of $X$ repeatedly, we consider it as taking smaller and smaller subsets of $\H(X)$.
\end{rem}

\subsection{Wright's construction}\label{subsec: Wright's construction}

In this section, we review the construction of a cobornologous Lipschitz map from a CAT(0) cube complex to its quotient in \cite{Wri12}. Throughout Section \ref{subsec: Wright's construction}, let $X$ be a CAT(0) cube complex of dimension $D \in \NNo$, let $\H(X)$ be the set of all hyperplanes of $X$, and fix a vertex $x_0 \in X^{(0)}$.

\begin{defn}\label{def:basic notations}
    Given $h \in \H(X)$, we define a choice of orientation for $h$ by $x_0 \in h^-$. Given $h,k \in \H(X)$,
    \begin{itemize}
        \item 
        $h$ and $k$ are said to \emph{cross} if $h^{(0)} \cap k^{(0)} \neq \emptyset$,
        \item
        we define $k\le h$ if $k^- \subset h^-$,
        \item 
        $h$ and $k$ are said to be \emph{opposite} if $h^+ \cap k^+ = \emptyset$.
    \end{itemize}
    For $x,y\in X^{(0)}$, define $\H(x,y)=\{h \in \H(X) \mid x \in h^- \Leftrightarrow y \in h^+\}$.
\end{defn}

\begin{rem}
    For any distinct $h,k \in \H(X)$, exactly one of the following four situations occurs; $h<k$, $k<h$, $h$ and $k$ are opposite, or $h$ and $k$ cross.
\end{rem}

The construction of the cobornologous Lipschitz map goes as follows. We first construct a nice subset $\K_c \subset \H(X)$ in Section \ref{subsec:Controlled colorings}. Next, in Section \ref{subsec:Interpolation of a quotient map}, we take the quotient of $X$ by $\K_c$ (see Definition \ref{def:Sageev-Roller duality}) and deform the quotient map $X\to X_{\K_c}$ in a larger CAT(0) cube complex $C(X_{\K_c})$ containing $X_{\K_c}$, which creates a contraction $\Psi_w\colon X \to C(X_{\K_c})$. Finally, we construct a projection $C(X_{\K_c}) \to X_{\K_c}$ and compose it with $\Psi_w$ in Section \ref{subsec:The Wright's projection theorem}.

\subsubsection{Controlled colorings}\label{subsec:Controlled colorings}

In Section \ref{subsec:Controlled colorings}, we will review the construction of controlled colorings, which is defined in Definition \ref{def:coloring}. Definition \ref{def:H^d_n} and Definition \ref{def:rank vector} below correspond to \cite[Definition 2.5, Definition 2.15]{Wri12}. They're presented slightly differently from \cite{Wri12} but define the same partition of $\H(X)$, because taking quotient of $X$ by the Sageev-Roller duality preserves transversality and inclusion of halfspaces.

\begin{defn}\label{def:H^d_n}
    Let $d \in \NN\setminus\{1\}$ and $\K \subset \H(X)$. For each $n \in \NNo$, we define $\K^d_{\ge n} \subset \H(X)$ by induction on $n$ as follows;
    \begin{align*}
        \K^d_{\ge 0} &= \K, \\
        \K^d_{\ge n+1} &= \{h \in \K^d_{\ge n} \mid \exists\,h_1,\cdots,h_d \in \K^d_{\ge n},\,\text{$\forall\,i\neq j$, $h_i$ and $h_j$ cross and } \forall\,i,\,h_i< h \}.
    \end{align*}
    For $n \in \NNo$, we define $\K^d_n \subset \H(X)$ by $\K^d_n = \K^d_{\ge n} \setminus \K^d_{\ge n+1}$, which defines the partition $\K = \bigsqcup_{n = 0}^\infty \K^d_n$.
\end{defn}

We get finer and finer partitions of $\H(X)$ by applying Definition \ref{def:H^d_n} repeatedly as follows.

\begin{defn}\label{def:H^(d_1,cdots,d_k)}
    For $k \in\NN$, $(d_1,\cdots,d_k) \in (\NN\setminus\{1\})^k$, and $(n_1,\cdots,n_k) \in (\NNo)^k$, define $\H^{(d_1,\cdots,d_k)}_{(n_1,\cdots,n_k)} \subset \H(X)$ by induction on $k$ as follows;
    \begin{align*}
        \H^{(d_1)}_{(n_1)} &= \H(X)^{d_1}_{n_1},\\
        \H^{(d_1,\cdots,d_{k+1})}_{(n_1,\cdots,n_{k+1})} &= \Big(\H^{(d_1,\cdots,d_k)}_{(n_1,\cdots,n_k)}\Big)^{d_{k+1}}_{n_{k+1}}.
    \end{align*}
\end{defn}

\begin{rem}
    Fix $k \in \NN$ and $(d_1,\cdots,d_k) \in (\NN\setminus\{1\})^k$, then we have
    \begin{align*}
        \H(X) = \bigsqcup_{(n_1,\cdots,n_k) \,\in\, (\NNo)^k}\H^{(d_1,\cdots,d_k)}_{(n_1,\cdots,n_k)}.
    \end{align*}
    Hence, for each $h \in \H(X)$, there exists a unique sequence $(n_1,\cdots,n_k) \in (\NNo)^k$ such that $h \in \H^{(d_1,\cdots,d_k)}_{(n_1,\cdots,n_k)}$.
\end{rem}

\begin{defn}\label{def:rank vector}
    When $D \ge 2$, a map $\rank \colon \H(X) \to (\NNo)^{D-1}$ is defined by
    \begin{align*}
        \rank(h) = (n_1,\cdots,n_{D-1}),\text{ where $h \in \H^{(D,D-1,\cdots,2)}_{(n_1,n_2,\cdots,n_{D-1})}$}.
    \end{align*}
    When $D \le 1$, we define a map $\rank \colon \H(X) \to \NNo$ by $\rank(h) = 0$ for any $h \in \H(X)$. The sequence $\rank(h)$ is called the \emph{rank vector of} $h$.
\end{defn}

\begin{rem}
The reason we define the rank vector differently in the case $D \le 1$ is because \cite[Proposition 2.14]{Wri12} is not true for $d=1$. This is because when $X$ is a tree, for any $k,h \in \H(X)$ satisfying $k<h$, the $1$-corner $k^+$ contains $h$ but doesn't contain $k$ (see \cite{Wri12} for relevant terminologies).
\end{rem}

\begin{defn}
    Let $k\in\NN$. The \emph{lexicographic order} $\le_{\mathrm{lex}}$ on $(\NNo)^k$ is defined as follows. For $(n_1,\cdots,n_k), (m_1,\cdots,m_k) \in (\NNo)^k$,
    \begin{align*}
        (n_1,\cdots,n_k) <_{\mathrm{lex}} (m_1,\cdots,m_k) \iff \exists\, i \in \NN,\, \Big(\bigwedge_{j<i}n_j = m_j\Big) \,\wedge\, n_i < m_i.
    \end{align*}
\end{defn}

Definition \ref{def:rank maximal} is from \cite[Definition 1.1, Definition 2.18]{Wri12}.

\begin{defn}\label{def:rank maximal}
    Given $h \in \H(X)$, a hyperplane $k \in \H(X)$ is called a \emph{predecessor} of $h$ if $k<h$ and $\{k_1 \in \H(X) \mid k<k_1<h\} = \emptyset$. A predecessor $k \in \H(X)$ of $h$ is called $\rank$-\emph{maximal} if for any predecessor $k_1 \in \H(X)$ of $h$, we have $\rank(k_1) \le_{\mathrm{lex}} \rank(k)$.
\end{defn}

\begin{rem}\label{rem:predecessor exists}
    A predecessor of $h \in \H(X)$ exists if and only if $x_0 \notin h^{(0)}$. In particular, a $\rank$-maximal predecessor of $h \in \H(X)$ exists whenever $x_0 \notin h^{(0)}$ since there exist at most $D$ predecessors of $h$ by \cite[Lemma 1.2]{Wri12}.
\end{rem}

\begin{defn}\cite[Definition 2.19]{Wri12}\label{def:coloring}
A map $c \colon \H(X) \to \{0,1\}$ is defined by
\begin{align*}
    c(h)
    =
    \begin{cases}
        1& {\rm if~} \text{for any $\rank$-maximal precedessor $k \in \H(X)$ of $h$, $c(k) = 0$}, \\
        0& {\rm otherwise}.
    \end{cases}
\end{align*}
We define $\K_c \subset \H(X)$ by $\K_c = c^{-1}(0)$.
\end{defn}

\begin{rem}\label{rem:1st color is 1}
    When $h \in \H(X)$ satisfies $x_0 \in h^{(0)}$, we have $c(h) = 1$ since the first condition in Definition \ref{def:coloring} is vacuously true by Remark \ref{rem:predecessor exists}. When $D = 1$, we have $c(h) = 1$ if $d(x_0, h) \in 2\NNo$ and $c(h) = 0$ if $d(x_0, h) \in 2\NN-1$.
\end{rem}

Theorem \ref{thm:coloring} below is from \cite[Theorem 2.23]{Wri12} combined with the fact that the flatness of a CAT(0) cube complex of dimension $D$ is at most $D$. We don't define controlled colorings in this paper, because we only use it as a black box to apply \cite[Lemma 4.7]{Wri12}, which is actually hided in the proof of \cite[Theorem 4.9]{Wri12}. Readers are referred to \cite[Definition 2.3, Definition 2.9]{Wri12} for relevant definitions.

\begin{thm}\label{thm:coloring}
    The map $c \colon \H(X) \to \{0,1\}$ defined in Definition \ref{def:coloring} is a $3^{D-1}D$-controlled coloring of $\H(X)$. 
\end{thm}

\subsubsection{Interpolation of a quotient map}\label{subsec:Interpolation of a quotient map}

\begin{defn}\label{def:C(X)}
    For $\xi \in [0,1]^{\H(X)}$, define $\supp(\xi) = \{h \in \H(X) \mid \xi(h) \neq 0\}$. A CAT(0) cube complex $C(X) \subset [0,1]^{\H(X)}$ is defined by
    \begin{align*}
        C(X) = \{\xi \in [0,1]^{\H(X)} \mid \#\supp(\xi)<\infty\},
    \end{align*}
    where the $0$-skeleton of $C(X)$ is $C(X) \cap \{0,1\}^{\H(X)}$. Define $\iota \colon X^{(0)} \to C(X)$ by
    \begin{align*}
        \iota(x)&= 1_{\H(x_0,x)}, \text{ where $1_{\H(x_0,x)}(h)=1$ if $h \in \H(x_0,x)$ and $1_{\H(x_0,x)}(h)=0$ if $h \notin \H(x_0,x)$}
    \end{align*}
    (see Definition \ref{def:basic notations}). The map $\iota$ affinely extends to $\iota \colon X \to C(X)$.
\end{defn}

\begin{rem}\label{rem:iota is isom}
    By $C(X) \subset \ell^1 (\H(X))$, we put $\ell^1$-metric on $C(X)$. By \cite[Theorem 3.6]{Wri12}, $\iota \colon X \to C(X)$ is isometric, where we put $\ell^1$-metric on $X$ as well.
\end{rem}

See Definition \ref{def:coloring} for the map $c \colon\H(X)\to \{0,1\}$ and the set $\K_c$ below.

\begin{defn}
    Let $\ell \in \NN$. Define $w \colon \H(X)^2 \to [0,1]$ by
    \begin{align*}
        w(h,k) =
        \begin{cases}
        \frac{\ell}{\ell+1}    & {\rm if~} h = k\\
        \frac{1}{\ell+1}    & {\rm if~} h<k \,\wedge\, c(k)=1 \,\wedge\, \forall\,j \in \H(X),\, h<j<k \Rightarrow c(j) = 0\\
        0 & {\rm otherwise}.
        \end{cases}
    \end{align*}
Define $\Psi_w \colon C(X) \to C(X_{\K_c})$ as follows; for $\xi \in C(X)$ and $k \in \K_c$,
\begin{align}\label{eq:(Psixi)(k)}
    (\Psi_w\xi)(k) = \min\big\{1, \sum_{h \in \H(X)}w(k,h)\xi(h)\big\}.
\end{align}
\end{defn}

\begin{rem}\label{rem:Psi on C(H) well-defined}
    In \cite{Wri12}, $\Psi_w\xi$ in \eqref{eq:(Psixi)(k)} is defined only for $\xi \in \iota(X)$, but it's straightforward to see that \eqref{eq:(Psixi)(k)} is well-defined for any $\xi \in C(X)$. Indeed, for any $\xi \in C(X)$, we have $\supp(\Psi_w\xi) \subset \{k \in \K_c \mid \exists\,h \in \supp(\xi),\, k \le h\}$. This implies $\#\supp(\Psi_w\xi)< \infty$ by $\#\supp(\xi)< \infty$, hence $\Psi_w\xi \in C(X_{\K_c})$.
\end{rem}

\subsubsection{The Wright's projection theorem}\label{subsec:The Wright's projection theorem}

\begin{defn}
     We define $\H^{op}, \H^< \subset \H(X)^2$ by
    \begin{align*}
        \H^{op} &= \{(h,k) \in \H(X)^2 \mid \text{$h$ and $k$ are opposite}\}.\\
        \H^< &= \{(h,k) \in \H(X)^2 \mid h < k\}.
    \end{align*}
\end{defn}

\begin{defn}
    Define maps $p^{op},\, p^< \colon [0,1]^2 \to [0,1]^2$ by
        \begin{align*}
        &p^{op}(x,y)=(p_1^{op}(x,y), p_2^{op}(x,y)) = 
        \begin{cases}
        (x-y,0)    & {\rm if~} x \ge y\\
        (0, y-x)    & {\rm if~} x \le y,
        \end{cases}\\[10pt]
        &p^<(x,y) = (p_1^<(x,y),p_2^<(x,y)\big) =
        \begin{cases}
        (x+y,0)    & {\rm if~} x + y \le 1\\
        (1, x+y-1)    & {\rm if~} x + y \ge 1.
        \end{cases}
    \end{align*}
\end{defn}

\begin{defn}\label{def:p_h,k^<, p_h,k^op}
    For $(h,k) \in \H(X)^2$ with $h \neq k$, we define maps $p_{h,k}^{op},\, p_{h,k}^< \colon C(X) \to C(X)$ as follows; for $\xi\in C(X)$,
    \begin{align*}
        \big((p_{h,k}^{op}\xi)(h), (p_{h,k}^{op}\xi)(k)\big) = p^{op}(\xi(h),\xi(k)) \text{ and $(p_{h,k}^{op}\xi)(j) = \xi(j)$ for any $j \in \H(X)\setminus\{h,k\}$}
    \end{align*}
    and
    \begin{align*}
        \big((p_{h,k}^<\xi)(h),(p_{h,k}^<\xi)(k)\big) = p^<(\xi(h),\xi(k)) \text{ and $(p_{h,k}^<\xi)(j) = \xi(j)$ for any $j \in \H(X)\setminus\{h,k\}$}.
    \end{align*}
\end{defn}

\begin{defn}
    Let $S$ be a set and $\{p_\lambda\}_{\lambda \in \Lambda}$ be a family of maps $p_\lambda \colon S \to S$. For $s \in S$, define $\Lambda \hy\supp(s) \subset \Lambda$ by $\Lambda\hy\supp(s) = \{\lambda \in \Lambda \mid p_\lambda(s) \neq s\}$. If a total order $\prec$ is defined on $\Lambda$, then for a finite subset $F = \{\lambda_1\prec \cdots \prec \lambda_m\} \subset \Lambda$, the map $p_F \colon S \to S$ is defined by
    \begin{align*}
        p_F(s) = p_{\lambda_m}\circ \cdots \circ p_{\lambda_1}(s).
    \end{align*}
    If in addition $\#\Lambda\hy\supp(s)<\infty$ for any $s\in S$, then the map $P_{(\Lambda,\prec)} \colon S \to S$ is defined by
    \begin{align*}
        P_{(\Lambda,\prec)}(s) = p_{\Lambda\hy\supp(s)}(s).
    \end{align*}
\end{defn}

\begin{lem}\label{lem:Lemma 3.5 of Wri12}{\rm \cite[Lemma 3.5]{Wri12}}
    Let $S$ be a set, $(\Lambda,\prec)$ be a totally ordered set, and $\{p_\lambda\}_{\lambda \in \Lambda}$ be a family of maps $p_\lambda \colon S \to S$. Suppose that (1) and (2) below hold.
    \item[(1)]
    For any $s \in S$, we have $\#\Lambda\hy\supp(s)<\infty$.
    \item[(2)]
    For any $s \in S$ and the $\prec$-least element $\lambda$ in $\Lambda\hy\supp(s)$, we have
    \begin{align*}
        \Lambda\hy\supp(p_\lambda(s)) \subset \Lambda\hy\supp(s)\setminus \{\lambda\}.
    \end{align*}
    Then, for any $s \in S$ and any finite subset $F \subset \Lambda$ satisfying $\Lambda\hy\supp(s) \subset F$, we have 
    \begin{align*}
        p_F(s) = p_{\Lambda\hy\supp(s)}(s) ~~~and~~~ \Lambda\hy\supp (p_F(s)) = \emptyset.
    \end{align*}
\end{lem}

\begin{rem}\label{rem:H^op-supp(xi), H^<-supp(xi) finite}
    Note that $\big\{p_{h,k}^{op}\big\}_{(h,k) \in \H^{op}}$ and $\big\{p_{h,k}^<\big\}_{(h,k) \in \H^<}$ are families of maps from $C(X)$ to $C(X)$. For any $\xi \in C(X)$, we have
    \begin{align*}
        \H^{op}\hy\supp(\xi) &=\{(h,k) \in \H^{op} \mid \xi(h) \neq 0 \,\wedge\, \xi(k) \neq 0 \} \subset \{(h,k) \in \H^{op} \mid h,k \in \supp(\xi)\},\\
        \H^<\hy\supp(\xi) &= \{(h,k) \in \H^< \mid \xi(h) \neq 1 \,\wedge\, \xi(k) \neq 0 \} \subset \{(h,k) \in \H^< \mid k \in \supp(\xi) \,\wedge\, h<k\}.
    \end{align*}
    In particular, both $\H^{op}\hy\supp(\xi)$ and $\H^<\hy\supp(\xi)$ are finite for any $\xi \in C(X)$.
\end{rem}

Theorem \ref{thm:supp of p^op, p^<} was proved in the proof of \cite[Theorem 3.6]{Wri12}.

\begin{thm}\label{thm:supp of p^op, p^<}
    The following hold.
    \begin{itemize}
        \item[(a)]
        For any $\xi \in C(X)$ and $(h,k) \in \H^{op}$, we have $\H^{op}\hy\supp(p_{h,k}^{op} \xi) \subset \H^{op}\hy\supp(\xi) \setminus \{(h,k)\}$.
        \item[(b)]
        For any $\xi \in C(X)$ and any $(h,k) \in \H^<\hy\supp(\xi)$ satisfying $$d_X(h^{(0)},k^{(0)}) = \max_{(h_1,k_1) \in \H^<\hy\supp(\xi)} d_X(h_1^{(0)},k_1^{(0)}),$$ we have $\H^<\hy\supp(p_{h,k}^< \xi) \subset \H^<\hy\supp(\xi) \setminus \{(h,k)\}$.
        \item[(c)]
        Suppose that total orders $\prec_1$ on $\H^{op}$ and $\prec_2$ on $\H^<$ satisfy the condition below;
        \begin{align*}
            \forall\, (h_1,k_1),(h_2,k_2) \in \H^<,\, d_X(h_1^{(0)},k_1^{(0)}) < d_X(h_2^{(0)},k_2^{(0)}) \Rightarrow (h_2,k_2) \prec_2 (h_1,k_1).
        \end{align*}
        Then, the map $P \colon C(X) \to C(X)$ defined by $P = P_{(\H^<,\prec_2)}\circ P_{(\H^{op},\prec_1)}$ satisfies (1) and (2) below.
        \begin{itemize}
            \item[(1)]
            $P(C(X)) \subset \iota(X)$ and $\forall\,\xi \in \iota(X),\, P\xi = \xi$.
            \item[(2)] 
            $P$ is contractive i.e. $\forall\, \xi_1,\xi_2 \in C(X),\, \|P\xi_1-P\xi_2\|_1 \le \|\xi_1,-\xi_2\|_1$.
        \end{itemize}
        \end{itemize}
\end{thm}

\subsection{Descriptive set theory}\label{sec:Descriptive set theory}

In this section, we recall descriptive set theory. Readers are referred to \cite{Kec95,Anu22,Kec25} for further details of descriptive set theory and Borel equivalence relations.

\begin{defn}\label{def:eq ref}
    Let $X$ be a set. An \emph{equivalence relation on} $X$ is a reflexive, symmetric, and transitive subset of $X^2$. Let $E$ be an equivalence relation on $X$. For $x,y\in X$, we denote $x\, E \,y$ when $(x,y)\in E$. For $A\subset X$, we define $[A]_E = \{y\in X\mid \exists\,x \in A,\, x\,E\, y\}$. For $x \in X$, we also denote $[\{x\}]_E$ by $[x]_E$. A subset $A\subset X$ is called a \emph{transversal} (resp. \emph{partial transversal}) for $E$ if $|A\cap[x]_E| = 1$ (resp. $|A\cap[x]_E| \le 1$) for any $x \in X$. A map $s \colon X\to X$ is called a \emph{selector for} $E$ if $s(x) \in [x]_E$ for any $x \in X$ and $x \,E\, y \iff s(x)=s(y)$ for any $x,y \in X$. An equivalence relation $E$ is called \emph{countable} (resp. \emph{finite}) if for any $x \in X$, $[x]_E$ is countable (resp. finite).
\end{defn}

\begin{defn}\label{def:graph}
    Let $X$ be a set. A \emph{graph on} $X$ is an anti-reflexive and symmetric subset of $X^2$. Let $G$ be a graph on $X$. We call a connected component of $G$ a $G$-\emph{component} and denote by $E_G^X$ the equivalence relation on $X$ defined by $G$-components. A graph is called \emph{countable} if $E_G^X$ is countable. We denote by $q_G^X\colon X \to X/E_G^X$ the quotient map. We denote by $d_G \colon X^2 \to [0,\infty]$ the metric on $X$ induced by $G$ (i.e. given $x,y \in X$, if $x\,E_G^X\, y$, then $d_G(x,y)$ is the graph metric of the $G$-component containing $x$ and $y$. If $\neg(x\, E_G^X\,y)$, then $\d_G(x,y) = \infty$).
\end{defn}

\begin{defn}
    A \emph{Polish space} is a topological space that is separable and completely metrizable. A measurable space $(X,\mathcal{B})$ is called a \emph{standard Borel space} if there exists a Polish topology $\mathcal{O}$ on $X$ such that $\mathcal{B}$ is the smallest $\sigma$-algebra containing $\mathcal{O}$. 
\end{defn}

\begin{rem}\label{rem:union of sbs}
    If $(X_n)_{n \in \NN}$ are standard Borel spaces, then the disjoint union $\bigsqcup_{n\in \NN}X_n$ becomes a standard Borel space by defining $A \subset \bigsqcup_{n\in \NN}X_n$ to be Borel if and only if $A \cap X_n$ is Borel in $X_n$ for any $n \in \NNo$. We will often shorten ``Borel measurable" to ``Borel".
\end{rem}

\begin{defn}\label{def:Borel eq rel}
    Let $X$ be a standard Borel space. An equivalence relation $E$ (resp. a graph $G$) on $X$ is called \emph{Borel} if $E\subset X^2$ (resp. $G\subset X^2$) is Borel in $X^2$. A Borel equivalence relation $E$ on $X$ is called \emph{smooth} if there exists a standard Borel space $Y$ and a Borel map $f \colon X \to Y$ such that for any $x,y \in X$, $x\, E \, y \iff f(x) = f(y)$. A \emph{Borel median graph on} $X$ is a Borel graph on $X$ such that every $G$-component is a median graph.
\end{defn}

\begin{rem}\label{rem:Borel subset of std}
    Any Borel subset of a standard Borel space is a standard Borel space. Hence, a Borel graph $G$ itself is a standard Borel space since $G$ is a Borel subset of $X^2$.
\end{rem}

\begin{rem}\label{rem:finite Borel eq}
    Any finite Borel equivalence relation on a standard Borel space is smooth.
\end{rem}

\begin{rem}\label{rem:quotient is std}
    For any CBER $E$ on a standard Borel space $X$, $E$ is smooth if and only if the quotient $X/E$ is a standard Borel space.
\end{rem}

We will often abbreviate ``countable Borel equivalence relation" to CBER. Lemma \ref{lem:smooth CBER} below is a well-known fact (see \cite[Proposition 2.12]{Kec25}, \cite[Section 20]{Anu22})

\begin{lem}\label{lem:smooth CBER}
    Let $E$ be a countable Borel equivalence relation (CBER) on a standard Borel space $X$. Then, the following are equivalent.
    \begin{itemize}
        \item[(1)]
        $E$ is smooth.
        \item[(2)]
        There exists a Borel transversal $A\subset X$ for $E$.
        \item[(3)]
        There exists a Borel selector $s\colon X \to X$ for $E$.
    \end{itemize}
\end{lem}

\begin{rem}\label{rem:Borel selector}
    In Lemma \ref{lem:smooth CBER} (2), the map $X \to A$ sending $x \in X$ to a unique point $y\in A \cap [x]_E$ is a Borel selector for $E$. In Lemma \ref{lem:smooth CBER} (3), the set $s(X)$ is a Borel transversal for $E$.
\end{rem}

Theorem \ref{thm:countable to one} below will be used throughout this paper for the case when the section $A_x$ is a countable set, which is $K_\sigma$. In Theorem \ref{thm:countable to one} (see \cite[Theorem 18.18, Theorem 35.43]{Kec95}, \cite[Corollary 13.7]{Anu22}), for $A \subset X\times Y$ and $x \in X$, we define $A_x \subset Y$ by
\begin{align}
    A_x=\{y \in Y \mid (x,y) \in A\}.
\end{align}

\begin{thm}[Arsenin-Kunugui]\label{thm:countable to one}
    Let $Y$ be a Polish space, $X$ a standard Borel space, and $A \subset X \times Y$ Borel such that $A_x$ is $K_\sigma$ (i.e. countable union of compact sets) in $Y$ for any $x \in X$. Then, the set $\mathrm{proj}_X(A)\,(=\{x \in X\mid \exists\,y \in Y,\,(x,y) \in A\})$ is Borel.
\end{thm}

\begin{rem}\label{rem:E_G^X and d_G are Borel}
    Let $X$ be a standard Borel space and $G \subset X^2$ be a countable Borel graph. By Theorem \ref{thm:countable to one}, the map $d_G\colon X^2 \to [0,\infty]$ is Borel and $E_G^X$ is a CBER.
\end{rem}

\begin{defn}
     Let $(X,\rho)$ be an extended metric space i.e. $\rho\colon X^2 \to [0,\infty]$ satisfies the axioms of metrics. The equivalence relation $E_\rho$ on $X$ is defined as follows; for any $x,y \in X$, $x\, E_\rho \, y \iff \rho(x,y)<\infty$. Given $U \subset X$ and $r > 0$, the equivalence relation $\mathfrak{F}_r(U)$ on $U$ is defined as follows; for any $x,y \in U$,
    \begin{align*}
        x \,\mathfrak{F}_r(U)\,y \iff \exists\, x_0,\cdots,x_n \in U,\,x_0=x \,\wedge\, x_n = y \,\wedge\, \wedge_{i=1}^n\,\rho(x_{i-1}, x_i) \le r.
    \end{align*}
    When $X$ is a standard Borel space and an extended metric $\rho \colon X^2 \to [0,\infty]$ is Borel measurable, $(X,\rho)$ is called a \emph{Borel extended metric space}.
\end{defn}

\begin{rem}
    Countable Borel graphs become a Borel extended metric space with their graph metric.
\end{rem}

The condition in Definition \ref{def:Borel asymptotic dimension} is from \cite[Lemma 3.1 (2)]{CJM23} (see also \cite[Lemma 2.4 (1)]{CJM23}).

\begin{defn}\cite[Definition 3.2]{CJM23}\label{def:Borel asymptotic dimension}
Let $(X,\rho)$ be a Borel extended metric space such that $E_\rho$ is countable. The \emph{Borel asymptotic dimension of} $(X,\rho)$, denoted $\mathrm{asdim}_{\mathbf{B}}(X,\rho)$, is the least $d \in \NNo$ such that for every $r > 0$, there exist Borel sets $U_0,U_1,\cdots,U_d \subset X$ such that $X=\bigcup_{i=0}^d U_i$ and $\exists\, M>0, \forall \,i,\, \forall\, (x,y) \in \mathfrak{F}_r(U_i),\, \rho(x,y)\le M$ (i.e. $\mathfrak{F}_r(U_i)$ is uniformly bounded for every $i \in \{0,\cdots,d\}$). If no such $d$ exists, then we define $\mathrm{asdim}_{\mathbf{B}}(X,\rho) = \infty$.
\end{defn}

\begin{rem}\label{rem:smooth and finite Borel asy dim}
    A metric space with infinite asymptotic dimension has infinite Borel asymptotic dimension but is smooth since all points are equivalent to one another. On the other hand, the orbit equivalence relation induced by the action of a non-elementary hyperbolic group on its Gromov boundary has finite Borel asymptotic dimension by \cite[Theorem A]{NV23}, but it is not smooth since the action is minimal (see \cite[Corollary 7.4.3 (ii)]{DSU17} and \cite[Corollary 20.19]{Anu22})
\end{rem}

\section{Gluing cubes to a Borel median graph}\label{sec:Gluing cubes to a Borel median graph}

In this section, we introduce a Borel extended metric space $\widetilde X$ by gluing cubes to a Borel median graph $(X,G)$. We start by preparing notations.

\begin{defn}\label{def:S_n}
Given $n \in \NNo$, we define the $\ell^1$-metric $d$ on $[0,1]^n$, that is, for $\alpha=(\alpha_1,\cdots,\alpha_n),\beta=(\beta_1,\cdots,\beta_n) \in [0,1]^n$, we have $d(\alpha,\beta) = \sum_{i=1}^n |\alpha_i - \beta_i|$. We denote by $\S_n$ the isometry group of $([0,1]^n,d)$. For convenience, we define $[0,1]^0$, $(0,1)^0$, and $\{0,1\}^0$ to be a singleton. Also, define $\S_0$ to be the trivial group $\{1\}$.
\end{defn}

\begin{rem}\label{rem:S_n is finite}
    We can see $g\{0,1\}^n = \{0,1\}^n$ for any $g \in \S_n$ and that $\S_n$ is embedded into the permutation group of $2^n$ points. Hence, $\S_n$ is finite for any $n \in \NNo$.
\end{rem}

\begin{defn}\label{def:widetilde X}
    Let $X$ be a standard Borel space and $G \subset X^2$ be a Borel median graph. We denote by $\widetilde X$ the (disjoint) union of the CAT(0) complexes whose 1-skeletons are all the $G$-components (note $(\widetilde{X})^{(0)}=X$). Also, denote by $d_{\widetilde X}$ the $\ell^1$-metric on $\widetilde X$. For $n \in \NN$, define $C_n \subset X^{\{0,1\}^n}$ by
    \begin{align*}
        C_n
        =
        \{(x_\alpha)_{\alpha \in \{0,1\}^n} \in X^{\{0,1\}^n} \mid \forall\, \alpha,\beta \in \{0,1\}^n,\, d(\alpha,\beta)=1 \iff (x_\alpha, x_\beta) \in G\}.
    \end{align*}
    For convenience, we define $C_0$ by $C_0 = X$. The group $\S_n$ acts on $C_n$ by permuting the coordinates, that is, for $g \in \S_n$ and $\vec{x} = (x_\alpha)_{\alpha \in \{0,1\}^n} \in C_n$, we have $(g\vec x)_\alpha = x_{g^{-1}\alpha}$. For $n \in \NNo$ and a $n$-cube $c \subset \widetilde X$ whose $0$-skeleton $c^{(0)}$ coincides with $\vec x = (x_\alpha)_{\alpha \in \{0,1\}^n} \in C_n$ as a set (i.e. $c^{(0)} = \{x_\alpha \mid \alpha \in \{0,1\}^n\}$), we define the unique isometry $F_{\vec x} \colon [0,1]^n \to c$ by mapping each $\alpha \in \{0,1\}^n$ to $x_\alpha$.
\end{defn}

\begin{rem}
    For example, if $X=\{0,1,2,3\}$ and $G=\{(0,1),(1,0),(2,3),(3,2)\}$, then $\widetilde X = [0,1]\sqcup [2,3]$, that is, $\widetilde X$ is the disjoint union of two closed intervals. We will consider $X$ as a subset of $\widetilde X$.
\end{rem}

\begin{rem}\label{rem:B_n,C_n Borel}
    For any $n \in \NNo$, the set $C_n$ is Borel since $G$ is Borel. Also, the orbit equivalence relation $E^{C_n}_{\S_n}$ induced by the action $\S_n \act C_n$ is smooth by Remark \ref{rem:finite Borel eq} and Remark \ref{rem:S_n is finite}.
\end{rem}

Next, we turn $\widetilde X$ into a standard Borel space by defining a $\sigma$-algebra.

\begin{defn}\label{def:sigma-algebra on widetilde X}
    Let $X$ be a standard Borel space and $G \subset X^2$ be a Borel median graph. For $n \in \NNo$, let $P_n \subset C_n$ be a Borel transversal for $E^{C_n}_{\S_n}$. Note $P_0 =C_0 = X$ by $\S_0=\{1\}$. Set $\mathcal{P} = (P_n)_{n \in \NNo}$. Define the map $F_\mathcal{P} \colon \bigsqcup_{n\in\NNo}P_n \times (0,1)^n \to \widetilde X$ as follows (see Definition \ref{def:widetilde X} for $F_{\vec x}$); 
    \begin{align}\label{eq:F_mathcal P}
    \forall\,n \in \NNo,\,\forall\,(\vec x, \vec t) \in P_n\times (0,1)^n,\, F_\mathcal{P}(\vec x, \vec t) = F_{\vec x}(\vec t).  
    \end{align}
    Here, for any $x \in P_0 \times(0,1)^0 = X$, we have $F_\mathcal{P}(x) = x \in \widetilde X$. Note that $F_\mathcal{P}$ is bijective and $\bigsqcup_{n\in\NNo}P_n \times (0,1)^n$ is a standard Borel space (see Remark \ref{rem:union of sbs}). Thus, $\widetilde X$ becomes a standard Borel space by defining the $\sigma$-algebra $\sigma_{\mathcal{P}}$ on $\widetilde X$ that makes $F_\mathcal{P}$ Borel isomorphic.
\end{defn}

In Lemma \ref{lem:siga_P is independent} below, we verify that the $\sigma$-algebra $\sigma_{\mathcal{P}}$ is canonical. In Section \ref{sec: Adding cubes to a Borel median graph (smooth version)}, we will introduce another way to define the same $\sigma$-algebra in a special case.

\begin{lem}\label{lem:siga_P is independent}
    Let $X$ be a standard Borel space and $G \subset X^2$ be a Borel median graph. Let $\mathcal{P} = (P_n)_{n \in \NNo}$ and $\mathcal{Q} = (Q_n)_{n \in \NNo}$ be such that $P_n$ and $Q_n$ are Borel transversals for $E^{C_n}_{\S_n}$ for every $n \in \NNo$. Then, we have $\sigma_\mathcal{P} = \sigma_\mathcal{Q}$.
\end{lem}

\begin{proof}
    Let $F_\mathcal{P} \colon \bigsqcup_{n\in\NNo}P_n \times (0,1)^n \to \widetilde X$ and $F_\mathcal{Q} \colon \bigsqcup_{n\in\NNo}Q_n \times (0,1)^n \to \widetilde X$ be the maps defined from $\mathcal{P}$ and $\mathcal{Q}$, respectively, as in Definition \ref{def:sigma-algebra on widetilde X}. Since both $P_n$ and $Q_n$ are Borel isomorphic to $C_n/\S_n$, there exists a Borel isomorphism $f_n \colon P_n \to Q_n$ for each $n \in \NNo$ such that for any $\vec x \in P_n$, there exists $g \in \S_n$ such that $f_n(\vec x) = g\vec x$. Given $n \in \NNo$ and $g \in \S_n$, define $P_{n,g} \subset P_n$ by
    \begin{align*}
        P_{n,g} = \{\vec x \in P_n \mid f_n(\vec x) = g\vec x \}.
    \end{align*}
    We have $P_n = \bigsqcup_{g \in \S_n} P_{n,g}$ and $P_{n,g}$ is Borel since $f_n$ and $g$ are Borel maps. Also, for any $(\vec x, \vec t) \in P_{n,g} \times (0,1)^n$, we have $F_\mathcal{Q}^{-1}\circ F_\mathcal{P} (\vec x, \vec t) = (f_n(\vec x), g\vec t)$. Hence, $F_\mathcal{Q}^{-1}\circ F_\mathcal{P}$ is Borel on $P_{n,g} \times (0,1)^n$ for any $n \in \NNo$ and $g \in \S_n$. Thus, $F_\mathcal{Q}^{-1}\circ F_\mathcal{P}$ is a Borel map. Since we can show that $F_\mathcal{P}^{-1}\circ F_\mathcal{Q}$ is Borel in the same way, we have $\sigma_\mathcal{P} = \sigma_\mathcal{Q}$.
\end{proof}

\begin{rem}\label{rem:sigma_X}
    By Lemma \ref{lem:siga_P is independent}, the $\sigma$-algebra $\sigma_\mathcal{P}$ in Definition \ref{def:sigma-algebra on widetilde X} is independent of the choice of Borel transversals $\mathcal{P}$. Hence, we will denote $\sigma_\mathcal{P}$ by $\sigma_{\widetilde X}$ and $(\widetilde{X},\sigma_{\widetilde X})$ is a standard Borel space. We can see that $\sigma_{\widetilde X}$ coincides the $\sigma$-algebra generated by the subsets
    \begin{align*}
        \{F_{\vec x}(\vec t)\mid \vec x \in A, \vec t \in B\}
    \end{align*}
    where $n$ varies over $\NNo$, $A$ varies over Borel partial transversals for $E_{\S_n}^{C_n}$, and $B$ varies over Borel subsets of $(0,1)^n$. Also, $X$ is a Borel subset of $\widetilde X$ and the $\sigma$-algebra on $X$ coincides with the $\sigma$-algebra induced by $\sigma_{\widetilde X}$.
\end{rem}

Finally, we show that $\widetilde X$ is a Borel extended metric space.

\begin{lem}\label{lem:ell1 metric on tilde X is Borel}
    Let $X$ be a standard Borel space and $G \subset X^2$ be a Borel median graph. Then, the following hold.
    \begin{itemize}
        \item[(1)]
        For any $n \in \NNo$, the map $C_n\times [0,1]^n \ni (\vec x,\vec t) \mapsto F_{\vec x}(\vec t) \in \widetilde X$ is Borel.
        \item[(2)]
        The $\ell^1$-metric $d_{\widetilde X} \colon (\widetilde X)^2 \to [0,\infty]$ is Borel.
    \end{itemize}
\end{lem}

\begin{proof}
    As in Definition \ref{def:sigma-algebra on widetilde X}, fix Borel transversals $\mathcal{P} = (P_0)_{n \in \NNo}$ for $E^{C_n}_{\S_n}$ for each $n \in \NNo$ and let $F_\mathcal{P} \colon \bigsqcup_{n \in \NNo}P_n \times (0,1)^n \to \widetilde X$ be the map defined by \eqref{eq:F_mathcal P}.

    (1) Let $k \in\{0,\cdots,n\}$, $1 \le i_1<\cdots<i_{n-k} \le n$, and $(\e_i)_{i=1}^{n-k} \in \{0,1\}^{n-k}$. Define the $k$-cube $D \subset [0,1]^n$ by $D=\{(t_1,\cdots,t_n) \mid \forall\,\ell \in \{1,\cdots,n-k\},\, t_{i_\ell}=\e_\ell\}$. Fix an isometry $\phi\colon [0,1]^k \to D$. It's enough to show that the map $F'\colon C_n\times (0,1)^k \to P_k\times (0,1)^k$ defined by $F'(\vec x,\vec t) = F_\mathcal{P}^{-1} \circ F_{\vec x}(\phi(\vec t))$ is Borel. 
    
    For any $\vec x=(x_\alpha)_{\alpha \in \{0,1\}^n} \in C_n$, there exists $g\in \S_k$ such that $g(x_{\phi(\beta)})_{\beta \in \{0,1\}^k} \in P_k$. For $g \in \S_k$, we define $C_g \subset C_n$ by $C_g = \{(x_\alpha)_{\alpha \in \{0,1\}^n} \in C_n \mid g(x_{\phi(\beta)})_{\beta \in \{0,1\}^k} \in P_k\}$. The set $C_g$ is Borel since $P_k$ is Borel. For any $(\vec x,\vec t) \in C_g\times (0,1)^k$, we have
    \begin{align*}
        F'(\vec x,\vec t) = (g(x_{\phi(\beta)})_{\beta \in \{0,1\}^k}, g\vec t).
    \end{align*}
    Hence, $F'$ is Borel on $C_g\times (0,1)^k$ for any $g \in \S_k$. By $C_n = \bigsqcup_{g \in \S_k} C_g$, the map $F'$ is Borel.
    
    (2) For $\vec x=(x_\alpha)_{\alpha \in \{0,1\}^k} \in C_k$, where $k \in \NNo$, we define $\{\vec x\} \subset X$ by 
    \begin{align*}
        \{\vec x\} = \{x_\alpha \mid \alpha \in \{0,1\}^k\}.
    \end{align*}
    Given $(\vec x_1, \vec t_1) \in P_{k_1}\times(0,1)^{k_1}$, $(\vec x_2, \vec t_2) \in P_{k_2}\times(0,1)^{k_2}$, and $\vec y \in P_m$, where $k_1,k_2,m \in \NNo$, that satisfy $\{\vec x_1\} \cup \{\vec x_2\} \subset \{\vec y\}$, there exist $\vec s_1,\vec s_2 \in [0,1]^m$ such that $F_{\vec y}(\vec s_i)=F_{\vec x_i}(\vec t_i)$ for every $i\in \{1,2\}$ since $F_{\vec x_1}([0,1]^{k_1})$ and $F_{\vec x_2}([0,1]^{k_2})$ are faces of the $m$-cube $F_{\vec y}([0,1]^m)$ in $\widetilde X$. We define $f_{\vec x_1,\vec x_2,\vec y}(\vec t_1,\vec t_2) \in \RR$ by
    \begin{align*}
        f_{\vec x_1,\vec x_2,\vec y}(\vec t_1,\vec t_2) = |\vec s_1-\vec s_2|_1.
    \end{align*}
    Fix $k_1,k_2,m \in \NNo$, then the set $V \subset P_{k_1}\times(0,1)^{k_1} \times P_{k_2}\times(0,1)^{k_2} \times P_m \times \RR$ defined by
    \begin{align*}
       &(\vec x_1, \vec t_1,\vec x_2, \vec t_2,\vec y, u) \in V \iff \\
       &\{\vec x_1\} \cup \{\vec x_2\} \subset \{\vec y\} \,\wedge\,\big[\exists\,\vec s_1,\vec s_2 \in [0,1]^m,\, \wedge_{i=1}^2\,F_{\vec y}(\vec s_i)=F_{\vec x_i}(\vec t_i) \,\wedge\,|\vec s_1-\vec s_2|_1=u \big].
    \end{align*}
    is Borel by Theorem \ref{thm:countable to one} and Lemma \ref{lem:ell1 metric on tilde X is Borel} (1). Hence, the map $(\vec x_1, \vec t_1,\vec x_2, \vec t_2,\vec y)\mapsto f_{\vec x_1,\vec x_2,\vec y}(\vec t_1,\vec t_2)$ is Borel on the set of points satisfying $\{\vec x_1\} \cup \{\vec x_2\} \subset \{\vec y\} $.
    
    Let $r>0$. It is enough to show that $\{(a,b) \in (\widetilde X)^2 \mid d_{\widetilde X}(a,b)<r\}$ is Borel to show Borel measurability of $d_{\widetilde X}$. For any $(a,b) \in (\widetilde X)^2$, we have
    \begin{align*}
        &d_{\widetilde X}(a,b) <r \iff\\
        &\exists\, n \in \NN,\, \exists\,k_0,\cdots,k_n,m_1,\cdots,m_n \in \NNo,\,\\
        &\exists\,(\vec x_0, \vec t_0) \in P_{k_0}\times(0,1)^{k_0},\cdots,\exists\,(\vec x_n, \vec t_n) \in P_{k_n}\times(0,1)^{k_n},\,\exists\, \vec y_1 \in P_{m_1},\cdots,\,\exists\,\vec y_n \in P_{m_n},\\
        &F_\mathcal{P}(\vec x_0, \vec t_0)=a \,\wedge\,F_\mathcal{P}(\vec x_n, \vec t_n)=b \,\wedge\,\wedge_{i=1}^n\, \{\vec x_{i-1}\} \cup \{\vec x_i\} \subset \{\vec y_i\}\,\wedge\,\sum_{i=1}^n f_{\vec x_{i-1},\vec x_i,\vec y_i}(\vec t_{i-1},\vec t_i)<r.
    \end{align*}
    Fix $n \in \NN$ and $k_0,\cdots,k_n,m_1,\cdots,m_n \in \NNo$ and set $Y=P_{k_0} \times\dots\times P_{k_n}$, $Z=P_{m_1} \times\dots\times P_{m_n}$, and $W=(0,1)^{k_0}\times\cdots\times(0,1)^{k_n}$. Also, fix a Polish topology on $Y\times Z$ corresponding to its $\sigma$-algebra. We define $A \subset (\widetilde X)^2\times Y \times Z \times W$ by
    \begin{align*}
        &(a,b,\vec x_0,\cdots,\vec x_n,\vec y_1,\cdots, \vec y_n, \vec t_0,\cdots, \vec t_n) \in A \iff \\
        &F_\mathcal{P}(\vec x_0, \vec t_0)=a \,\wedge\,F_\mathcal{P}(\vec x_n, \vec t_n)=b \,\wedge\,\wedge_{i=1}^n\, \{\vec x_{i-1}\} \cup \{\vec x_i\} \subset \{\vec y_i\}\,\wedge\,\sum_{i=1}^n f_{\vec x_{i-1},\vec x_i,\vec y_i}(\vec t_{i-1},\vec t_i)<r.
    \end{align*}
    We can see that $A$ is Borel by Borel measurability of the map $f_{\vec x_{i-1},\vec x_i,\vec y_i}(\vec t_{i-1},\vec t_i)$ shown above. For $(a,b) \in (\widetilde X)^2$, $\mathbf{x} \in Y$, $\mathbf{y}\in Z$, $\vec t_0 \in (0,1)^{k_0}$, and $\vec t_n \in (0,1)^{k_n}$, we define
    \begin{align*}
        &A_{(a,b)}=\{(\mathbf{x'},\mathbf{y'},\mathbf{t'}) \in Y \times Z \times W \mid (a,b,\mathbf{x'},\mathbf{y'},\mathbf{t'}) \in A\},\\
        &B_{(a,b)} =\{(\mathbf{x'},\mathbf{y'},\vec t'_0, \vec t'_n) \in Y \times Z \times (0,1)^{k_0}\times(0,1)^{k_n}\mid \\
        &~~~~~~~~~~~~~~\exists \,(\vec t'_1,\cdots,\vec t'_{n-1})\in (0,1)^{k_1}\times\cdots\times(0,1)^{k_{n-1}},\,
        (a,b,\mathbf{x'},\mathbf{y'}, \vec t'_0, \vec t_1,\cdots,\vec t_{n-1},\vec t'_n) \in A \},\\
        &A_{(a,b,\mathbf{x},\mathbf{y},\vec t_0, \vec t_n)}=\{(\vec t'_1,\cdots,\vec t'_{n-1}) \in (0,1)^{k_1}\times\cdots\times(0,1)^{k_{n-1}} \mid \\
        &~~~~~~~~~~~~~~~~~~~~~~~~~~~~~~~~~~~~~~~~~~~~~~~~~~~~~~~~~~~~~~~~~~~~~~~(a,b,\mathbf{x},\mathbf{y}, \vec t_0, \vec t'_1,\cdots,\vec t'_{n-1},\vec t_n) \in A\}.
    \end{align*}
    Since $E_G^X$ is countable, $B_{(a,b)}$ is countable for any $(a,b) \in (\widetilde X)^2$. Given $(a,b,\mathbf{x},\mathbf{y},\vec t_0, \vec t_n) \in (\widetilde X)^2 \times Y \times Z \times (0,1)^{k_0}\times(0,1)^{k_n}$, the set $A_{(a,b,\mathbf{x},\mathbf{y},\vec t_0, \vec t_n)}$ is open in $(0,1)^{k_1}\times\cdots\times(0,1)^{k_{n-1}}$, hence $K_\sigma$. Hence, $A_{(a,b)}$ is $K_\sigma$ in $Y \times Z \times W$ for any $(a,b) \in (\widetilde X)^2$ by
    \begin{align*}
        A_{(a,b)} = \bigcup_{(\mathbf{x},\mathbf{y},\vec t_0, \vec t_n) \in B_{(a,b)}} \{(\mathbf{x},\mathbf{y},\vec t_0, \vec t_n)\} \times A_{(a,b,\mathbf{x},\mathbf{y},\vec t_0, \vec t_n)}.
    \end{align*}
    Thus, the set $\mathrm{proj}_{(\widetilde X)^2}(A)$ is Borel by Theorem \ref{thm:countable to one}. Since $\{(a,b) \in (\widetilde X)^2 \mid d_{\widetilde X}(a,b)<r\}$ is the countble union of all Borel sets $\mathrm{proj}_{(\widetilde X)^2}(A)$ over $n \in \NN$ and $k_0,\cdots,k_n,m_1,\cdots,m_n \in \NNo$, it is Borel.
\end{proof}

\begin{rem}
    In the same way as Lemma \ref{lem:ell1 metric on tilde X is Borel}, we can show that $\ell^p$-metric on $\widetilde X$ is Borel for any $p\in [1,\infty]$.
\end{rem}

\section{Borel version of the Sageev-Roller duality}\label{sec:Standard Borel space of hyperplanes}

In this section, we study a countable Borel median graph $G$ on a standard Borel space $X$ such that $E_G^X$ is smooth. The goal of this section is to prove Proposition \ref{prop:E_K smooth}. 

We first introduce the structure of $\sigma$-algebra to the set of hyperplanes of $G$, which turn it into a standard Borel space. By using the standard Borel space of hyperplanes, we introduce a quotient of $G$ and produce a new Borel median graph. This procedure can be considered as Borel version of the Sageev-Roller duality.

Essentially, the same notion as a standard Borel space of hyperplanes was also defined in \cite{CPTT23} under different conditions on Borel graphs (see \cite[Remark 4.5]{CPTT23}). Our formalization is tailored to our condition and convenient to introduce the Borel version of the Wright's construction later.

\begin{defn}
    Let $X$ be a standard Borel space and $G \subset X^2$ be a Borel median graph. We define an equivalence relation $E_\square^G \subset G^2$ as follows; for $(x,y), (z,w) \in G$, we define $(x,y)\,E_\square^G\, (z,w)$ if there exist $(x,y)=(x_0,y_0),(x_1,y_1)\cdots,(x_n,y_n)=(z,w) \in G$ such that for any $i \in \{1,\cdots,n\}$, one of $\{(x_{i-1},x_i),(y_{i-1},y_i)\} \subset G$, $(x_i,y_i) = (y_{i-1},x_{i-1})$, or $(x_i,y_i) = (x_{i-1},y_{i-1})$ holds.
\end{defn}

The equivalence classes of $E_\square^G$ can be used to turn the set of hyperplanes of $G$ into a standard Borel space by Lemma \ref{lem:std of hyperplanes} below. Recall that every connected component of a Borel graph is a CAT(0) cube complex.

\begin{lem}\label{lem:std of hyperplanes}
    Let $X$ be a standard Borel space and $G \subset X^2$ be a countable Borel median graph. Then, $E_\square^G \subset G^2$ is a CBER. Moreover, if $s\colon X\to X$ is a Borel selector for $E_G^X$, then there exists a Borel transversal $\H_s(G) \subset G$ for $E_\square^G$ satisfying ($\ast$) below.
    \begin{itemize}
    \item[{\rm($\ast$)}]
    For any $(x,y) \in \H_s(G)$, let $h \in \H\big([x]_{E_G^X}\big)$ be the hyperplane containing the edge $(x,y)$ of the $G$-component $[x]_{E_G^X}\big(=[y]_{E_G^X}\big)$, which is a CAT(0) cube complex (note $[x]_{E_G^X}=h^-\sqcup h^+$). Then, $x$ and $s(x)$ are in the same halfspace delimited by $h$ (i.e. either $\{x,s(x)\}\subset h^-$ or $\{x,s(x)\}\subset h^+$ holds).    
    \end{itemize}
    In particular, if $E_G^X$ is smooth, then $E_\square^G$ is smooth.
\end{lem}

\begin{proof}
     Since $E_G^X$ is countable, $E_\square^G$ is countable. Next, we will show that $E_\square^G$ is Borel. Define a relation $R_\square^G \subset G^2$ by
    \begin{align*}
        R_\square^G = \{((x,y),(z,w)) \in G^2 \mid \{(x,z),(y,w)\} \subset G \,\vee\, (z,w) = (y,x)\,\vee\, (z,w) = (x,y)\},
    \end{align*}
    then $R_\square^G$ is Borel in $G^2$ since $G$ is Borel. For each $n \in \NN$, define a set $R_n \subset G^{n+1}$ and a map $p_n \colon G^{n+1}\to G^2$ by
    \begin{align*}
        &R_n=\big\{(e_i)_{i=0}^n \in G^{n+1} \mid \forall\,i\ge 1,\, (e_{i-1},e_i) \in R_\square^G\big\},\\
        &p_n((e_i)_{i=0}^n)=(e_0,e_n).
    \end{align*}
    By countability of $E_G^X$ and Theorem \ref{thm:countable to one}, the set $p_n(R_n)$ is Borel in $G^2$. Hence, $E_\square^G$ is Borel by $E_\square^G = \bigcup_{n\in \NN}p_n(R_n)$.

    Finally, let $s \colon X \to X$ be a Borel selector for $E_G^X$. Define $\H_s(G) \subset G$ by
    \begin{align*}
        \H_s(G) = \big\{(x,y) \in G\mid \forall\, (x',y') \in [(x,y)]_{E_\square^G}, \, d_G(s(x), x) \le d_G(s(x), x')\big\}.
    \end{align*}
    Note that for any $\big((x,y),(x',y')\big) \in E_\square^G$, we have $s(x) = s(y) = s(x') = s(y')$. To show Borel measurability of $\H_s(G)$, we define a set $B \subset G^2$ and a map $p \colon G^2 \to G$ by
    \begin{align*}
            &B = \big\{\big((x,y),(x',y')\big) \in E_\square^G \mid d_G(s(x), x) > d_G(s(x), x')\big\},\\
            &p(e,f) = e.
    \end{align*}
    The set $B$ is Borel in $E_\square^G$ by Remark \ref{rem:E_G^X and d_G are Borel} and the restriction $p|_B$ is a countable-to-one Borel map since $E_\square^G$ is countable. Hence, $p(B)$ is Borel by Theorem \ref{thm:countable to one}. Hence, $\H_s(G)$ is Borel by $\H_s(G) = G \setminus p_1(B)$. By Remark \ref{rem:separate convex sets}, we can see that $\H_s(G)$ is a Borel transversal for $E_\square^G$ satisfying ($\ast)$. The last statement in Lemma \ref{lem:std of hyperplanes} follows from Lemma \ref{lem:smooth CBER}.
\end{proof}

In Lemma \ref{lem:std of hyperplanes}, $\H_s(G)$ is bijective to the (disjoint) union of the sets of hyperplanes of all $G$-components. Therefore, we name this set in the following way.

\begin{defn}\label{def:std of hyperplanes}
    Let $X$ be a standard Borel space and $G \subset X^2$ be a countable Borel median graph. Given a Borel selector $s\colon X\to X$ for $E_G^X$, we call a Borel transversal $\H_s(G) \subset G$ for $E_\square^G$ satisfying the condition $(\ast)$ in Lemma \ref{lem:std of hyperplanes} a \emph{space of hyperplanes of $G$ for $s$}.
\end{defn}

\begin{rem}
    Even if $E_G^X$ is not smooth, $E_\square^G$ can be smooth, for example, when every $E_\square^G$-class is finite, which is the condition considered in \cite{CPTT23}. In this case, we can define a standard Borel space of hyperplanes of $G$ by taking a Borel transversal for $E_\square^G$.
\end{rem}

\begin{rem}
    By considering $\bigcup_{(x,y)\in \H_s(G)}\{(x,y),(y,x)\}$, we can also define a space of halfspaces of $G$, but it is not necessary in this paper.
\end{rem}

\begin{rem}\label{rem:use of countable to one map}
    As in the proof of Lemma \ref{lem:std of hyperplanes}, it is a standard argument to show Borel measurability by using projections and applying Theorem \ref{thm:countable to one}. Therefore, we will often omit details of this argument in what follows.
\end{rem}

In Definition \ref{def:notions on H_s(G)}, we introduce various notions related to hyperplanes for Borel median graphs. Many of them correspond to Definition \ref{def:hyperplane}. We define the sets $\I^-$ and $\I^+$ just to simplify proofs. 

\begin{defn}\label{def:notions on H_s(G)}
    Let $X$ be a standard Borel space and $G \subset X^2$ be a countable Borel median graph such that $E_G^X$ is smooth. Let $s\colon X\to X$ be a Borel selector for $E_G^X$ and $\H_s(G)$ be a space of hyperplanes of $G$ for $s$. We define an equivalence relation $E_G^{\H_s(G)}$ on $\H_s(G)$ as follows; for any $h,k \in \H_s(G)$, $h\,E_G^{\H_s(G)}\,k$ if $h$ and $k$ are in the same $G$-component (i.e. when $h=(x,y)$ and $k=(z,w)$, $h\,E_G^{\H_s(G)}\,k \Leftrightarrow x\, E_G^X\, z$). Given $h=(x,y) \in \H_s(G)$, we define $h^{(0)} = \bigcup\{\{x',y'\} \mid(x',y') \in G,\, (x,y)\, E_\square^G\, h\} \subset [x]_{E_G^X}$ and define $h^-,h^+ \subset [x]_{E_G^X}$ to be the halfspaces in $[x]_{E_G^X}$ delimited by the hyperplane in $\H\big([x]_{E_G^X}\big)$ containing $h$ such that $s(x) \in h^-$ (note $[x]_{E_G^X}=h^-\sqcup h^+$, $x \in h^-$, and $y \in h^+$ by Definition \ref{def:std of hyperplanes}). We also call $h^-,h^+$ \emph{halfspaces delimited by $h$}. Given $x,y \in X$ with $x\,E_G^X\,y$ and $h \in \H_s(G)$, we say that $h$ \emph{separates $x$ and $y$} if $h$ is in the same $G$-component as $x$ and $y$ and we have $x\in h^-\iff y \in h^+$. For $x,y \in X$, we define $\H_s(G)(x,y)$ to be the set of all $h \in \H_s(G)$ that separates $x$ and $y$. We define $\I^-,\I^+ \subset X \times \H_s(G)$ by
    \begin{align*}
        \I^- = \{(x,h) \in X \times \H_s(G) \mid x \in h^-\} {\rm ~~~and~~~} \I^+ = \{(x,h) \in X \times \H_s(G) \mid x \in h^+\}.
    \end{align*}
    Given $\K \subset \H_s(G)$, we define an equivalence relation $E_\K^X \subset X^2$ as follows; for $x,y \in X$, $x \, E_\K^X \,y$ if $x \, E_G^X \,y$ and no $h \in \K$ separates $x$ and $y$. We denote the quotient $X / E_\K^X$ by $X_\K$ and define a graph $G_\K \subset (X_\K)^2$ as follows; for $x,y \in X$, $([x]_{E_\K^X},[y]_{E_\K^X}) \in G_\K$ if $x \, E_G^X \,y$ and there exists a unique $h \in \K$ that separates $x$ and $y$. We denote by $q_\K^X \colon X\to X_\K$ the quotient map.
\end{defn}

\begin{rem}\label{rem:X_K is median graphs}
    In other words, given $\K \subset \H_s(G)$, the graph $(X_\K, G_\K)$ is a (disjoint) union of median graphs obtained by taking the quotient of every $G$-component by the hyperplanes belonging to $\K$ and the $G$-component.
\end{rem}

To prove Proposition \ref{prop:E_K smooth}, we prepare auxiliary lemmas, Lemma \ref{lem:notions on Borel median graph are Borel} and Lemma \ref{lem:seq of transversal}. Lemma \ref{lem:notions on Borel median graph are Borel} shows basic properties of some of the notions defined above.

\begin{lem}\label{lem:notions on Borel median graph are Borel}
    Let $X$ be a standard Borel space and $G \subset X^2$ be a countable Borel median graph such that $E_G^X$ is smooth. Let $s\colon X\to X$ be a Borel selector for $E_G^X$ and $\H_s(G)$ be a space of hyperplanes of $G$ for $s$. Then, the following hold. 
    \begin{itemize}
        \item[(1)]
        The equivalence relation $E_G^{\H_s(G)}$ on $\H_s(G)$ is a smooth CBER. Moreover, the map $\H_s(G)/E_G^{\H_s(G)} \to X/E_G^X$ defined from $q_G^{\H_S(G)}\colon\H_s(G)\ni (x,y) \mapsto q_G^X(x) \in X/E_G^X$ is injective and Borel.
        \item[(2)] 
        The sets $\I^-,\I^+ \subset X \times \H_s(G)$ are Borel.
    \end{itemize}
\end{lem}

\begin{proof}
    (1) We can see that $E_G^{\H_s(G)}$ is a CBER since $E_G^X$ is a CBER. The map $q_G^{\H_S(G)} \colon \H_s(G) \to X/E_G^X$ defined in the statement is Borel and satisfies $h\,E_G^{\H_s(G)}\,k \Leftrightarrow q_G^{\H_S(G)}(h)=q_G^{\H_S(G)}(k)$ for any $h,k \in \H_s(G)$. Hence, $E_G^{\H_s(G)}$ is smooth.

    By Lemma \ref{lem:smooth CBER}, take a Borel transversal $A$ for $E_G^{\H_s(G)}$. Let $\psi \colon \H_s(G)/E_G^{\H_s(G)} \to A$ be the Borel isomorphism defined by taking the inverse of the restriction to $A$ of the quotient map $\H_s(G) \to \H_s(G)/E_G^{\H_s(G)}$. The map $\H_s(G)/E_G^{\H_s(G)} \to X/E_G^X$ defined from $q_G^{\H_S(G)}$ is equal to $q_G^{\H_S(G)}\circ\psi$, hence Borel and injective.
    
    (2) For any $(a,(x,y)) \in X \times \H_s(G)$, we have $(a,(x,y)) \in \I^-$ if and only if there exists a path in the graph $G$ from $a$ to $x$ none of whose edges is in the same hyperplane as $(x,y)$. This implies
    \begin{align*}
        (a,(x,y)) \in \I^- \iff &\exists\,n \in \NN,\,\exists\,x_0,\cdots,x_n\in X,\, x_0=a \,\wedge\,x_n=x\\
        &~~~~\wedge\,\Big(\forall\,i \ge 1,x_{i-1}=x_i \vee \big[(x_{i-1},x_i) \in G \wedge \neg((x_{i-1},x_i)\,E_\square^G\, (x,y)) \big]\Big).
    \end{align*}
    Hence, $\I^-$ is Borel by the countability of $E_G^X$ and Theorem \ref{thm:countable to one} (see Remark \ref{rem:use of countable to one map}). Borel measurability of $\I^+$ follows in the same way.
\end{proof}

\begin{rem}\label{rem:map from H to X/E_G}
    The natural map $\H_s(G)/E_G^{\H_s(G)} \to X/E_G^X$ in Lemma \ref{lem:notions on Borel median graph are Borel} (1) is not necessarily surjective. Indeed, if a $G$-component is a singleton, then this component has no hyperplane.
\end{rem}

Lemma \ref{lem:seq of transversal} below is straightforward, but we write down the proof for completeness.

\begin{lem}\label{lem:seq of transversal}
    Let $X$ be a standard Borel space and $E \subset X^2$ be a CBER. If $E$ is smooth, then there exists a sequence of Borel transversals $(A_n)_{n \in \NN}$ for $E$ such that $X = \bigcup_{n \in \NN}A_n$. As a corollary, there exists a sequence of Borel partial transversals $(B_n)_{n \in \NN}$ for $E$ such that $X = \bigsqcup_{n \in \NN}B_n$
\end{lem}

\begin{proof}
    Since $E$ is smooth, there exists a Borel transversal $A$ for $E$ by Lemma \ref{lem:smooth CBER}. By Feldman-Moore Theorem, there exists a countable group $H$ and a Borel action $H \act X$ such that $E = E_H^X$, where $E_H^X$ is the orbit equivalence relation of the action $H\act X$. Let $H = \{g_n \mid n \in \NN\}$. Define $A_n \subset X$ by $A_n = g_n A$, then each $A_n$ is a Borel transversal for $E$ by $E = E_H^X$ and we have $X = \bigcup_{n \in \NN}A_n$.

    The corollary follows by defining $B_n$ by $B_n = A_n \setminus \bigcup_{i=1}^{n-1} A_i$.
\end{proof}

\begin{prop}\label{prop:E_K smooth}
    Let $X$ be a standard Borel space and $G \subset X^2$ be a countable Borel median graph such that $E_G^X$ is smooth. Let $s\colon X\to X$ be a Borel selector for $E_G^X$ and $\H_s(G)$ be a space of hyperplanes of $G$ for $s$. Then, for any Borel subset $\K \subset \H_s(G)$, the following (1)-(5) hold.
    \begin{itemize}
        \item[(1)]
        $E_\K^X$ is a smooth CBER, hence $X_\K$ is a standard Borel space.
        \item[(2)]
        $G_\K$ is a countable Borel median graph on $X_\K$
        \item[(3)]
        $(q_\K^X\circ s)(X) \subset X_\K$ is a Borel transversal for $E_{G_\K}^{X_\K}$. Hence, $E_{G_\K}^{X_\K}$ is a smooth CBER.
        \item[(4)]
        Let $s_\K\colon X_\K \to (q_\K^X\circ s)(X)$ be the natural Borel selector for $E_{G_\K}^{X_\K}$ as in Remark \ref{rem:Borel selector}. Then, the subset $(q_\K^X \times q_\K^X)(\K) \subset (X_\K)^2$ is a space of hyperplanes of $G_\K$ for $s_\K$, which we denote by $\H_{s_\K}(G_\K)$, and $(q_\K^X \times q_\K^X)|_\K \colon \K \to \H_{s_\K}(G_\K)$ is a Borel isomorphism.
        \item[(5)] 
        The natural map $\Phi_\K^X \colon X/E_G^X \to X_\K/E_{G_\K}^{X_\K}$ sending each $G$-component to its quotient is Borel isomorphic.
    \end{itemize}
\end{prop}

\begin{proof}
    (1) Since $E_G^{\H_s(G)}$ is smooth by Lemma \ref{lem:notions on Borel median graph are Borel} (1), there exists a sequence of Borel transversals $(A_n)_{n \in \NN}$ for $E_G^{\H_s(G)}$ such that $\H_s(G) = \bigcup_{n \in \NN}A_n$ by Lemma \ref{lem:seq of transversal}. For each $n \in \NN$, define a map $\psi_n \colon X \to \{0,1\}$ by 
    \begin{align}\label{eq:E_K^X smooth}
    \psi_n(x)=
        \begin{cases}
        1 & {\rm if~} \exists\,h \in \K \cap A_n,\,x \in h^-,\\
        0 & {\rm otherwise}.
        \end{cases}
    \end{align}
    Note that $\psi_n$ is well-defined since $A_n$ is a transversal for $E_G^{\H_s(G)}$. For any $x \in X$, we have $x \in \psi_n^{-1}(1) \iff \exists\, h \in \K \cap A_n,\, (x,h) \in \I^-$. Hence, $\psi_n^{-1}(1)$ is Borel by countability of $E_G^X$, Lemma \ref{lem:notions on Borel median graph are Borel} (2), and Theorem \ref{thm:countable to one}. Thus, $\psi_n$ is a Borel map for any $n \in \NN$. Define a map $\psi \colon X \to X \times 2^\NN$ by $\psi(x) = (s(x), \{\psi_n(x)\}_{n \in \NN})$. By $\K = \bigcup_{n \in \NN}\K \cap A_n$, $\psi$ is a Borel map such that
    \begin{align*}
        x \, E_\K^X \,y \iff x \,E_G^X\,y \,\wedge\, \forall\,n \in \NN,\, \psi_n(x) = \psi_n(y) \iff \psi(x) = \psi(y)
    \end{align*}
    for any $x,y \in X$. Thus, $E_\K^X$ is a smooth CBER. Hence, $X_\K$ is a standard Borel space by Remark \ref{rem:quotient is std}.
    
    (2) Let $\psi_n$ be as in \eqref{eq:E_K^X smooth}, then for any $(x,y) \in X^2$, we have $(x,y) \in (q_\K^X \times q_\K^X)^{-1}(G_\K)$ if and only if
    \begin{align*}
        x \,E_G^X\,y \,\wedge\, \exists\,(x_1,y_1) \in \K,\, \Big[x \,E_G^X\,x_1 \,\wedge\, \big(\forall\, n \in\NN,\, (x_1,y_1) \notin \K \cap A_n \Leftrightarrow \psi_n(x) = \psi_n(y) \big)\Big].
    \end{align*}
    Hence, $(q_\K^X \times q_\K^X)^{-1}(G_\K)$ is Borel in $X^2$ by countability of $E_G^X$ and Theorem \ref{thm:countable to one}. Thus, $G_\K$ is Borel. By countability of $E_G^X$ and Remark \ref{rem:X_K is median graphs}, $G_\K$ is a countable Borel median graph.

    (3) Since $s(X)$ is a Borel transversal for $E_G^X$ and the restriction $q_\K^X|_{s(X)} \colon s(X) \to X_\K$ is injective, the set $(q_\K^X\circ s)(X) \subset X_\K$ is a Borel transversal for $E_{G_\K}^{X_\K}$ by Theorem \ref{thm:countable to one} and Remark \ref{rem:X_K is median graphs}. Hence, $E_{G_\K}^{X_\K}$ is a smooth CBER by Lemma \ref{lem:smooth CBER}.

    (4) By Remark \ref{rem:quotient of CAT(0) cc} and Remark \ref{rem:X_K is median graphs}, the restriction $(q_\K^X \times q_\K^X)|_\K \colon \K \to (G_\K)^2$ is injective and the set $(q_\K^X \times q_\K^X)(\K)$ is a transversal for $E_\square^{G_\K}$. By this injectivity and Theorem \ref{thm:countable to one}, $(q_\K^X \times q_\K^X)(\K)$ is Borel and the map $(q_\K^X \times q_\K^X)|_\K$ is Borel isomorphic. We can see that $(q_\K^X \times q_\K^X)(\K)$ is a space of hyperplanes of $G_\K$ for $s_\K$ since $\H_s(G)$ is a space of hyperplanes for $s$.
    
    (5) Since $s(X)$ and $(q_\K^X\circ s)(X)$ are Borel transversals for $E_G^X$ and $E_{G_\K}^{X_\K}$ respectively, the restrictions $q_G^X|_{s(X)} \colon s(X) \to X/E_G^X$ and $q_{G_\K}^{X_\K}|_{(q_\K^X\circ s)(X)} \colon (q_\K^X\circ s)(X) \to X_\K/E_{G_\K}^{X_\K}$ are Borel isomorphic (see Definition \ref{def:graph}). The map $q_\K^X \colon s(X) \to (q_\K^X\circ s)(X)$ is Borel isomorphic by its injectivety and Theorem \ref{thm:countable to one}. Hence, $\Phi^X_\K$ is Borel isomorphic by $\Phi_\K^X = q_{G_\K}^{X_\K} \circ q_\K^X \circ (q_G^X)^{-1}$.
\end{proof}

\begin{rem}\label{rem:H_{G_K}=(q times q)(K)}
In Proposition \ref{prop:E_K smooth} (4), since $s(X)$ and $(q_\K^X\circ s)(X)$ are Borel transversals for $E_G^X$ and $E_{G_\K}^{X_\K}$ respectively, we can fix a unique point in $s(X)$ (resp. $(q_\K^X\circ s)(X)$) for each $G$-component (resp. $G_\K$-component) and define the order on $\H_s(G)$ (resp. $\H_{s_\K}(G_\K)$) as in Definition \ref{def:rank maximal}. Then, for any $h,k \in \K$, we have $h<k \iff (q_\K^X \times q_\K^X)(h) < (q_\K^X \times q_\K^X)(k)$.
\end{rem}

\section{Glueing cubes to a Borel median graph (smooth case)}\label{sec: Adding cubes to a Borel median graph (smooth version)}

In this section, we continue the study a countable Borel median graph $G$ on a standard Borel space $X$ such that $E_G^X$ is smooth. We will introduce a Borel extended metric space $C(G)$, which is larger than $\widetilde X$ (see Definition \ref{def:widetilde X}) and play an important role in the Borel version of the Wright's construction. By embedding $\widetilde X$ into $C(G)$, we will get another description of the Borel structure on $\widetilde X$.

\begin{defn}\label{def:C(G)}
    Let $X$ be a standard Borel space and $G \subset X^2$ be a countable Borel median graph such that $E_G^X$ is smooth. Let $s\colon X\to X$ be a Borel selector for $E_G^X$ and $\H_s(G)$ be a space of hyperplanes of $G$ for $s$. Given $\xi \in [0,1]^{\H_s(G)}$, we define $\supp(\xi)$ by
    \begin{align*}
        \supp(\xi) = \{h \in \H_s(G) \mid \xi(h) \neq 0\}.
    \end{align*}
    Define $C(G) \subset X/E_G^X \times [0,1]^{\H_s(G)}$ by (see Lemma \ref{lem:notions on Borel median graph are Borel} (1) for $q_G^{\H_s(G)}$)
    \begin{align*}
        C(G) = \big\{(a,\xi) \in X/E_G^X \times [0,1]^{\H_s(G)} \,\big|\, \#\supp(\xi)< \infty {\rm~and~} q_G^{\H_s(G)}(\supp(\xi))\subset\{a\}\big\}.
    \end{align*}
    Define $d \colon C(G)^2 \to [0,\infty]$ by
    \begin{align*}
        d((a,\xi), (b,\zeta))=
        \begin{cases}
            \sum_{h \in \H_s(G)}|\xi(h)-\zeta(h)| & {\rm if~} a=b\\
            \infty & {\rm if~} a \neq b.
        \end{cases}
    \end{align*}
    Given Borel transversals $\mathcal{A}=(A_n)_{n \in \NN}$ for $E_G^{\H_s(G)}$ satisfying $\H_s(G) = \bigcup_{n \in \NN}A_n$, define a map $\psi_\mathcal{A}\colon C(G) \to X/E_G^X \times [0,1]^\NN$ by
    \begin{align}\label{eq:C(H_G)}
        \psi_\mathcal{A}(a, \xi) = \big(a, \{\xi(A_n \cap (q_G^{X})^{-1}(a))\}_{n \in \NN}\big).
    \end{align}
     Since the map $\psi_\mathcal{A}$ is injective by $\H_s(G) = \bigcup_{n \in \NN}A_n$, we define a $\sigma$-algebra $\sigma_\mathcal{A}$ on $C(G)$ by the restriction of the $\sigma$-algebra on $X/E_G^X \times [0,1]^\NN$ to $\psi_\mathcal{A}(C(G))$. Define $\iota \colon X \to C(G)$ by
    \begin{align}
        \iota(x)&= (q_G^X(x), 1_{\H_s(G)(s(x),x)})
    \end{align}
    (see Definition \ref{def:notions on H_s(G)} for $\H_s(G)(s(x),x)$). The map $\iota$ extends to the map $\widetilde\iota \colon \widetilde X \to C(G)$ affinely.
\end{defn}

\begin{rem}\label{rem:A for psi_A exists}
    The definition \eqref{eq:C(H_G)} is well-defined since $A_n \cap (q_G^X)^{-1}(a)$ is a singleton for any $n \in \NN$. Also, the sequence $\mathcal{A}$ of Borel transversals exists by Lemma \ref{lem:notions on Borel median graph are Borel} (1) and Lemma \ref{lem:seq of transversal}.
\end{rem}

\begin{rem}
    In other words, $C(G)$ and $\iota$ are obtained by applying the construction in Definition \ref{def:C(X)} to all $G$-components with base points belonging to $s(X)$. Readers are referred to the proof of Proposition \ref{lem: C(H_G) is sBc} (b) for the details on how to extend $\iota \colon X \to C(G)$ to $\widetilde X$ affinely.
\end{rem}

In Lemma \ref{lem:sbs of C(H_G) independent} below, we verify that the $\sigma$-algebra $\sigma_\mathcal{A}$ is canonical.

\begin{lem}\label{lem:sbs of C(H_G) independent}
    Let $X$ be a standard Borel space and $G \subset X^2$ be a countable Borel median graph such that $E_G^X$ is smooth. Let $s\colon X\to X$ be a Borel selector for $E_G^X$ and $\H_s(G)$ be a space of hyperplanes of $G$ for $s$. Then, for any sequences of Borel transversals $\mathcal{A}=(A_n)_{n \in \NN}$ and $\mathcal{A'}=(A'_n)_{n \in \NN}$ for $E_G^{\H_s(G)}$ satisfying $\H_s(G) = \bigcup_{n \in \NN}A_n = \bigcup_{n \in \NN}A'_n$, we have $\sigma_{\mathcal{A}} = \sigma_{\mathcal{A'}}$.
\end{lem}

\begin{proof}
     Let $\psi_\mathcal{A}, \psi_\mathcal{A'}\colon C(G) \to X/E_G^X \times [0,1]^\NN$ be the maps defined in \eqref{eq:C(H_G)}. It's enough to show that $\psi_\mathcal{A'}\circ\psi_\mathcal{A}^{-1}\colon \psi_\mathcal{A}(C(G)) \to \psi_\mathcal{A'}(C(G))$ is isomorphic. Define $B \subset \big(X/E_G^X \times [0,1]^\NN\big)^2$ by $B = \{(\psi_\mathcal{A}(\eta), \psi_\mathcal{A'}(\eta)) \mid \eta \in C(G)\}$, then for any point $((a,\{b_n\}_{n \in \NN}),(a',\{b'_n\}_{n \in \NN}))$ in $\big(X/E_G^X \times [0,1]^\NN\big)^2$, we have
    \begin{align*}
        &((a,\{b_n\}_{n \in \NN}),(a',\{b'_n\}_{n \in \NN})) \in B \iff\\
        & a=a' \,\wedge\, \Big[(\forall\, n\in\NN,\,b_n=b'_n=0) \,\vee\, \exists\, k \in \NN,\,\exists\, h_1,\cdots,h_k\in \H_s(G),\, \exists\,t_1,\cdots,t_k \in (0,1],\\
        &~~~~~~~~~~\wedge_{i=1}^k\,q_G^{\H_s(G)}(h_i)=a \,\wedge\, \wedge_{i \neq j} \, h_i \neq h_j\\
        &~~~~~~~~~~\wedge~\Big(\forall\,n \in\NN,\, \wedge_{i=1}^k\,(h_i \in A_n \Rightarrow \,b_n = t_i) \,\wedge\, \big((\forall\,i,\,h_i \notin A_n) \Rightarrow \,b_n = 0\big) \Big)\\
        &~~~~~~~~~~\wedge~\Big(\forall\,n \in\NN,\, \wedge_{i=1}^k\,(h_i \in A'_n \Rightarrow \,b'_n = t_i) \,\wedge\, \big((\forall\,i,\,h_i \notin A'_n) \Rightarrow \,b'_n = 0\big) \Big)\Big].
    \end{align*}
    Hence, we can see that $B$ is Borel in $\big(X/E_G^X \times [0,1]^\NN\big)^2$ by countability of $E_G^X$ and Theorem \ref{thm:countable to one}. Here, we also used $\{t_i\}_{i=1}^k \subset \{b_n\}_{n \in \NN}$, which follows by $\H_s(G) = \bigcup_{n \in \NN}A_n$. By applying Theorem \ref{thm:countable to one} to the projections from $B$ to each coordinate, we can see that both $\psi_\mathcal{A'}\circ\psi_\mathcal{A}^{-1}$ and $\psi_\mathcal{A}\circ\psi_\mathcal{A'}^{-1}$ are measurable.
\end{proof}

\begin{rem}
By Lemma \ref{lem:sbs of C(H_G) independent}, the $\sigma$-algebra $\sigma_\mathcal{A}$ in Definition \ref{def:C(G)} is independent of the choice of Borel transversals $\mathcal{A}$. Hence, we will denote $\sigma_\mathcal{A}$ by $\sigma_{C(G)}$. 
\end{rem}

Lemma \ref{lem: C(H_G) is sBc} below provides another way to define $\sigma_{\widetilde X}$ (see Remark \ref{rem:sigma_X}).

\begin{lem}\label{lem: C(H_G) is sBc}
Let $X$ be a standard Borel space and $G \subset X^2$ be a countable Borel median graph such that $E_G^X$ is smooth. Let $s\colon X\to X$ be a Borel selector for $E_G^X$ and $\H_s(G)$ be a space of hyperplanes of $G$ for $s$. Then, the following hold.
\begin{itemize}
    \item[(1)]
    $(C(G), \sigma_{C(G)})$ is a standard Borel space.
    \item[(2)]
    The map $\widetilde\iota \colon \widetilde X \to \widetilde\iota (\widetilde X)$ is Borel isomorphic.
    \item[(3)]
    The map $d \colon C(G)^2 \to [0,\infty]$ is Borel.
\end{itemize}
\end{lem}

\begin{proof}
Let $\mathcal{A}=(A_n)_{n \in \NN}$ be Borel transversals for $E_G^{\H_s(G)}$ satisfying $\H_s(G) = \bigcup_{n \in \NN}A_n$ (see Remark \ref{rem:A for psi_A exists}) and let $\psi_\mathcal{A}$ be the map defined in \eqref{eq:C(H_G)}.

(1) By Remark \ref{rem:Borel subset of std}, it's enough to show that $\psi_\mathcal{A}(C(G))$ is Borel.  For any $(a,\{b_n\}_{n \in \NN})$ in $X/E_G^X \times [0,1]^\NN$, we have
\begin{align*}
    &(a,\{b_n\}_{n \in \NN}) \in \psi_\mathcal{A}(C(G)) \\
    &\iff
    (\forall\, n \in \NN,\, b_n = 0) \,\vee\, \exists\, k \in \NN,\,\exists\, h_1,\cdots,h_k\in \H_s(G),\, \exists\,t_1,\cdots,t_k \in (0,1],\\
    &~~~~~~~~~~\wedge_{i=1}^k\,q_G^{\H_s(G)}(h_i)=a \,\wedge\, \wedge_{i \neq j} \, h_i \neq h_j\\
    &~~~~~~~~~~\wedge~\Big(\forall\,n \in\NN,\, \wedge_{i=1}^k\,(h_i \in A_n \Rightarrow \,b_n = t_i) \,\wedge\, \big((\forall\,i,\,h_i \notin A_n) \Rightarrow \,b_n = 0\big) \Big).
\end{align*}
Hence, we can see that $\psi_\mathcal{A}(C(G))$ is Borel by Theorem \ref{thm:countable to one} in the same way as Lemma \ref{lem:sbs of C(H_G) independent}.

(2) As in Definition \ref{def:sigma-algebra on widetilde X}, fix Borel transversals $\mathcal{P} = (P_0)_{n \in \NNo}$ for $E^{C_n}_{\S_n}$ for each $n \in \NNo$ and let $F_\mathcal{P} \colon \bigsqcup_{n \in \NNo}P_n \times (0,1)^n \to \widetilde X$ be the map defined by \eqref{eq:F_mathcal P}. We will show that the map $F'$ defined by $F' = \psi_\mathcal{A}\circ\widetilde\iota \circ F_\mathcal{P}$ is Borel. 

Define $\mathbf{0}=(0,\cdots,0) \in \{0,1\}^n$. Given $\alpha = (\alpha(1),\cdots,\alpha(n)) \in \{0,1\}^n$, define $\alpha_i \in \{0,1\}^n$, for each $i \in \{1,\cdots,n\}$, by $\alpha_i(i)\neq \alpha(i)$ and $\alpha_i(j) = \alpha(j)$ for any $j\neq i$. Also, define $P_{n,\alpha} \subset P_n$ by
\begin{align*}
    P_{n,\alpha} = \{\vec x = (x_\beta)_{\beta \in \{0,1\}^n} \in P_n \mid d_X(s(x_{\mathbf{0}}), x_\alpha) = \min_{\beta \in \{0,1\}^n}d_X(s(x_{\mathbf{0}}), x_\beta) \}.
\end{align*}
We have $P_n = \bigsqcup_{\alpha \in \{0,1\}^n} P_{n,\alpha}$ by Remark \ref{rem:separate convex sets} and each $P_{n,\alpha}$ is Borel. Given a point $(\vec x, \vec t) \in P_{n,\alpha} \times (0,1)^n$, where $\vec t=(t_i)_{i=1}^n$, there exists unique $h_i \in \H_s(G)$ such that $h_i \,E_\square^G \, (x_\alpha, x_{\alpha_i})$ for every $i \in \{1,\cdots,n\}$. We have (see Definition \ref{def:notions on H_s(G)} for $\H_s(G)(s(x_\mathbf{0}),x_\alpha)$)
\begin{align*}
    \widetilde\iota \circ F_\mathcal{P} (\vec x, \vec t) = \Big(q_G^X(x_\mathbf{0}), 1_{\H_s(G)(s(x_\mathbf{0}),x_\alpha)} + \sum_{i = 1}^n u_i1_{h_i}\Big),{\rm ~where~}
    u_i = 
    \begin{cases}
        t_i & {\rm if~} \alpha(i) = 0\\
        1-t_i & {\rm if~} \alpha(i) = 1.
    \end{cases}
\end{align*}
Hence, for any $(\vec x, \vec t,a,\{b_m\}_{m \in \NN}) \in P_{n,\alpha} \times (0,1)^n \times (X/E_G^X) \times [0,1]^\NN$, we have
\begin{align*}
    &F'(\vec x, \vec t) = (a,\{b_m\}_{m \in \NN})\\
    &\iff
    (a,\{b_m\}_{m \in \NN}) \in \psi_\mathcal{A}(C(G)) \,\wedge\, q_G^X(x_\mathbf{0})=a \, \wedge\,\forall\,m \in \NN,\\
    &\big[(\exists\,h \in A_m,\,\wedge_{\beta \in \{0,1\}^n}\,(x_\beta,h) \in \I^+) \Rightarrow b_m = 1\big] \,\wedge\, \big[(\exists\,h \in A_m,\,\wedge_{\beta \in \{0,1\}^n}\,(x_\beta,h) \in \I^-)\Rightarrow b_m = 0\big] \\
    &\wedge \wedge_{i=1}^n \big[ (\exists h \in A_m, h \, E_\square^G \, (x_\alpha,x_{\alpha_i})) \Rightarrow b_m = u_i \big].
\end{align*}
(When $n=0$, the last line ``$\wedge \wedge_{i=1}^n \big( (\exists h \in A_m, h \, E_\square^G \, (x_\alpha,x_{\alpha_i})) \Rightarrow b_m = u_i \big)$" does not exist.) Hence, we can see that the set $\{ ((\vec x, \vec t), F'(\vec x, \vec t)) \mid (\vec x, \vec t) \in P_{n,\alpha} \times (0,1)^n \}$ is Borel by Theorem \ref{thm:countable to one}. Thus, $F'$ is Borel on $P_{n,\alpha} \times (0,1)^n$ for any $n \in \NNo$ and $\alpha \in \{0,1\}^n$ by applying Theorem \ref{thm:countable to one} to projections. Hence, $F'$ is Borel and also injective. This implies the statement by Theorem \ref{thm:countable to one}.

(3) Define $D \subset \psi_\mathcal{A}(C(G))^2 \times [0,\infty]$ by $D = \{(\psi_\mathcal{A}(\eta_1),\psi_\mathcal{A}(\eta_2),d(\eta_1,\eta_2)) \mid \eta_1,\eta_2 \in C(G)\}$. For any $((a,\{b_n\}_{n \in \NN}),(a',\{b'_n\}_{n \in \NN}),u) \in \psi_\mathcal{A}(C(G))^2 \times [0,\infty]$, we have
    \begin{align*}
        &((a,\{b_n\}_{n \in \NN}),(a',\{b'_n\}_{n \in \NN}),u) \in D \iff\\
        & (a\neq a' \,\wedge\, u =\infty) \,\vee\\
        &~~~~~~~~~~ a=a' \,\wedge\, \bigg[ \big((\forall\, n\in\NN,\,b_n=b'_n=0) \,\wedge\, u=0\big) \\
        &~~~~~~~~~~\vee\,  \exists \, k \in \NN,\,\exists\, h_1,\cdots,h_k\in \H_s(G),\, \exists\,t_1,\cdots,t_k, t'_1,\cdots,t'_k \in [0,1],\\
        &~~~~~~~~~~\Big[\wedge_{i=1}^k\,q_G^{\H_s(G)}(h_i)=a \,\wedge\, \wedge_{i \neq j} \, h_i \neq h_j \,\wedge\, \sum_{i=1}^k|t_i-t'_i|=u\\
        &~~~~~~~~~~\wedge~\Big(\forall\,n \in\NN,\, \wedge_{i=1}^k\,(h_i \in A_n \Rightarrow \,b_n = t_i \,\wedge\, b'_n = t'_i) \,\wedge\, \big((\forall\,i,\,h_i \notin A_n) \Rightarrow \,b_n = b'_n= 0\big)\Big)\Big]\bigg]
    \end{align*}
    We can see that $D$ is Borel by Theorem \ref{thm:countable to one}. Hence, $d \colon C(G)^2 \to [0,\infty]$ is Borel.
\end{proof}

\begin{rem}\label{rem:widetilde X identified}
    By Lemma \ref{lem: C(H_G) is sBc} (2), we will often identify $\widetilde X$ and $\widetilde \iota(\widetilde X)$.
\end{rem}

\section{Borel measurability of Wright's construction}\label{sec:Borel measurability of Wright's construction}

In this section, we introduce the Borel version of the Wright's construction and verify Borel measurability of various notions that appear in the construction.

Throughout Section \ref{sec:Borel measurability of Wright's construction}, let $X$ be a standard Borel space and $G \subset X^2$ be a countable Borel median graph such that $E_G^X$ is smooth. Let $s\colon X\to X$ be a Borel selector for $E_G^X$ and $\H_s(G)$ be a space of hyperplanes of $G$ for $s$. Also, suppose that there exists $D \in \NNo$ such that every $G$-component in $X$ is a CAT(0) cube complex of dimension at most $D$. For each $G$-component, fix a unique point $x_0$ in $s(X)$ for Wright's construction (see Section \ref{subsec: Wright's construction}).

Recall that for any $x\in X$, the set $\H_s(G) \cap \big([x]_{E_G^X}\big)^2$ is bijective to the set of all hyperplanes of $[x]_{E_G^X}$. Therefore, we will often identify them in what follows.

\begin{defn}
    Define $\H_s(G)^{cr}, \H_s(G)^<, \H_s(G)^{op} \subset \H_s(G)^2$ by (see Definition \ref{def:basic notations})
    \begin{align*}
        \H_s(G)^{cr} &= \{(h,k) \in \H_s(G)^2 \mid h \,E_G^{\H_s(G)}\, k \,\wedge\, \text{$h$ and $k$ cross}\}, \\
        \H_s(G)^< &= \{(h,k) \in \H_s(G)^2 \mid h \,E_G^{\H_s(G)}\, k \,\wedge\, h < k\}, \\
        \H_s(G)^{op} &= \{(h,k) \in \H_s(G)^2 \mid h \,E_G^{\H_s(G)}\, k \,\wedge\, \text{$h$ and $k$ are opposite}\}.
    \end{align*}    
\end{defn}

\begin{lem}\label{lem:Cr,Pa Borel}
    The sets $\H_s(G)^{cr}$, $\H_s(G)^<$, and $\H_s(G)^{op}$ are Borel.
\end{lem}

\begin{proof}
    For any $(h,k) \in \H_s(G)^2$, we have
    \begin{align*}
        (h,k) \in \H_s(G)^{cr} &\iff
        \exists\,(x_\alpha)_{\alpha \in \{0,1\}^2} \in C_2,\, (x_{(0,0)},x_{(1,0)}) \,E_\square^G\, h \,\wedge\, (x_{(0,0)},x_{(0,1)}) \,E_\square^G\, k,\\
        (h,k) \in \H_s(G)^< &\iff
        h \neq k \,\wedge\, \forall\, x \in X,\,(x,h) \in \I^- \Rightarrow (x,k) \in \I^-, \\
        (h,k) \in \H_s(G)^{op} &\iff 
        h \,E_G^{\H_s(G)}\, k \,\wedge\, \neg\big[ \exists\,x \in X,\, (x,h) \in \I^+ \,\wedge\, (x,k) \in \I^+\big].
    \end{align*}
    Hence, the statement follows by countability of $E_G^X$ and Theorem \ref{thm:countable to one}.
\end{proof}

\subsection{Controlled colorings}\label{subsec:Controlled colorings Borel ver}

This section corresponds to Section \ref{subsec:Controlled colorings}.

\begin{defn}\label{def:Borel controlled coloring}
    Given $m \in \NNo$, we define $X_m \subset X$ to be the set of all points $x \in X$ such that $[x]_{E_G^X}$ is a CAT(0) cube complex of dimension $m$. Also, we set $Y_m = \H_s(G) \cap (X_m)^2$ and define the \emph{rank vector function $\rank$ on} $Y_m$ by applying the construction in Definition \ref{def:rank vector} to each $G$-component. (Here, $\rank \colon Y_m \to (\NNo)^{m-1}$ if $m \ge 2$ and $\rank \colon Y_m \to \NNo$ if $m \le 1$.) We define the map $c \colon \H_s(G) \to \{0,1\}$ by applying the construction in Definition \ref{def:coloring} to each $G$-component.
\end{defn}

\begin{lem}\label{lem:c is Borel}
    The following hold.
    \begin{itemize}
        \item[(1)]
        For any $m \in \NNo$, the sets $X_m \subset X$ and $Y_m \subset \H_s(G)$ are Borel.
        \item[(2)]
        For any $m \in \NNo$, the rank vector function $\rank$ on $Y_m$ is Borel.
        \item[(3)]
        The map $c \colon \H_s(G) \to \{0,1\}$ is Borel.
    \end{itemize}
\end{lem}

\begin{proof}
    (1) Given $m \in \NNo$, for any $x \in X$, we have
    \begin{align*}
        x \in \bigsqcup_{i=m}^\infty X_i \iff &\exists\, (x_\alpha)_{\alpha \in\{0,1\}^m} \in C_m,\, (x,x_{(0,\cdots,0)}) \in E_G^X.
    \end{align*}
    The set $\bigsqcup_{i=m}^\infty X_i$ is Borel by countability of $E_G^X$ and Theorem \ref{thm:countable to one}. Hence, $X_m$ is Borel by $X_m = \bigsqcup_{i=m}^\infty X_i \setminus\bigsqcup_{i=m+1}^\infty X_i$, which implies that $Y_m$ is Borel.

    (2) For $d \in \NN\setminus\{1\}$, $n \in \NNo$, and a Borel subset $\K \subset \H_s(G)$, define $\K^d_{\ge n} \subset \H_s(G)$ by induction on $n$ as follows (see Definition \ref{def:H^d_n});
    \begin{align*}
        \K^d_{\ge 0} &= \K,\\
        \K^d_{\ge n+1} &= \{h \in \K^d_{\ge n} \mid \exists\,h_1,\cdots,h_d \in \K^d_{\ge n},\, \forall\,i\neq j,\,(h_i,h_j) \in \H_s(G)^{cr} \,\wedge\, \forall\,i,\,(h_i,h) \in \H_s(G)^<\}.
    \end{align*}
    We can show that $\K^d_{\ge n}$ is Borel for any $d \ge 2$ and $n \ge 0$ by induction on $n$ and by Theorem \ref{thm:countable to one} and Lemma \ref{lem:Cr,Pa Borel}. Hence, the set $\K^d_n \subset \H_s(G)$ defined by $\K^d_n = \K^d_{\ge n} \setminus \K^d_{\ge n+1}$ is also Borel. 
    
    For $k \in\NN$, $d_1,\cdots,d_k \in \NN\setminus\{1\}$, and $n_1,\cdots,n_k \in \NNo$, define $\H_s(G)^{(d_1,\cdots,d_k)}_{(n_1,\cdots,n_k)} \subset \H_s(G)$ by induction on $k$ as follows (see Definition \ref{def:H^(d_1,cdots,d_k)});
    \begin{align*}
        &\H_s(G)^{(d_1)}_{(n_1)} = \H_s(G)^{d_1}_{n_1},\\
        &\H_s(G)^{(d_1,\cdots,d_{k+1})}_{(n_1,\cdots,n_{k+1})} = \Big(\H_s(G)^{(d_1,\cdots,d_k)}_{(n_1,\cdots,n_k)}\Big)^{d_{k+1}}_{n_{k+1}}.
    \end{align*}
    By induction on $k$, $\H_s(G)^{(d_1,\cdots,d_k)}_{(n_1,\cdots,n_k)}$ is Borel for any $k \in\NN$, $d_1,\cdots,d_k \in \NN\setminus\{1\}$, and $n_1,\cdots,n_k \in \NNo$.
    
    When $m \ge 2$, for any $(n_1,\cdots,n_{m-1}) \in (\NNo)^{m-1}$, we have $(\rank|_{Y_m})^{-1}(n_1,\cdots,n_{m-1}) = Y_m \cap \H_s(G)^{(m,\cdots,2)}_{(n_1,\cdots,n_{m-1})}$. Hence, the map $\rank|_{Y_m}$ is Borel for any $m \ge 2$. When $m \le 1$, the map $\rank|_{Y_m}$ is constant by definition, hence Borel. 

    (3) For $r \in \NNo$, we define $Z_r = \{h = (x,y) \in \H_s(G) \mid d_X(s(x), h^{(0)}) = r\}$. For any $h=(x,y) \in \H_s(G)$, we have by Remark \ref{rem:separate convex sets},
    \begin{align*}
        h \in \bigsqcup_{i=r+1}^\infty Z_i 
        &\iff d_X(s(x), h^{(0)}) \ge r+1 \\
        &\iff \exists\,h_1,\cdots,h_{r+1} \in \H_s(G),\, \wedge_{i\neq j}\,h_i \neq h_j \,\wedge\, \wedge_{i=1}^{r+1}\, (h_i,h) \in \H_s(G)^<.
    \end{align*}
    Hence, $\bigsqcup_{i=r+1}^\infty Z_i$ is Borel for any $r \in \NNo$ by countability of $E_G^X$ and Theorem \ref{thm:countable to one}, hence so is $Z_r$ by $Z_r = \bigsqcup_{i=r}^\infty Z_i \setminus \bigsqcup_{i=r+1}^\infty Z_i$. Note $\H_s(G) = \bigsqcup_{i=0}^\infty Z_i$.
    
    To show that $c$ is Borel, it is enough to show that $c|_{Y_m} \colon Y_m \to \{0,1\}$ is Borel for any $m \in \{0,\cdots,D\}$. We'll show this by showing that $c|_{Y_m \cap Z_r} \colon Y_m\cap Z_r \to \{0,1\}$ is Borel for any $r \in \NNo$ by induction on $r$. The map $c|_{Y_m \cap Z_0}$ is Borel since we have $c|_{Y_m \cap Z_0} \equiv 1$ by Remark \ref{rem:1st color is 1}. 
    
    Next, let $r \in \NN$ and assume that $c|_{Z_{r'}}$ is Borel for any $r'<r$. We define subsets $\pred,\pmax \subset \H_s(G)^2$ by (see Definition \ref{def:rank maximal})
    \begin{align*}
        \pred &= \{(h,k) \in \H_s(G)^< \mid \forall\,k_1 \in \H_s(G), (h,k_1) \notin \H_s(G)^< \,\vee\, (k_1, k) \notin \H_s(G)^<\}, \\
        \pmax &= \{(h,k) \in \pred \mid \forall\, h_1 \in \H_s(G),\, (h_1,k) \in \pred \Rightarrow \rank(h_1) \le_{\mathrm{lex}} \rank(h)\}.
    \end{align*}
    The sets $\pred$ and $\pmax$ are Borel by Theorem \ref{thm:countable to one}, Lemma \ref{lem:Cr,Pa Borel}, and Lemma \ref{lem:c is Borel} (2). Given $(h,k) \in \pmax$, let $v \in X$ satisfy $h^{(0)}\cup k^{(0)} \subset [v]_{E_G^X}$, then we have $d_X(v,h^{(0)}) < d_X(v,k^{(0)})$. Hence, for any $k \in Y_m\cap Z_r$, we have (see Definition \ref{def:coloring})
    \begin{align*}
        k \in (c|_{Y_m\cap Z_r})^{-1}(0) \iff \exists\,h \in \big(c|_{\bigsqcup_{i=0}^{r-1} Y_m \cap Z_i}\big)^{-1}(1),\, (h,k) \in \pmax. 
    \end{align*}
    Note that $c|_{\bigsqcup_{i=0}^{r-1} Y_m \cap Z_i}$ is Borel by our assumption of induction. Hence, $c|_{Y_m\cap Z_r}$ is Borel by Theorem \ref{thm:countable to one} since $E_G^X$ is countable and $\pmax$ is Borel. Since $c|_{Y_m\cap Z_r}$ is Borel for any $r \ge 0$, the map $c|_{Y_m}$ is Borel.
\end{proof}

\subsection{Interpolation of a quotient map}\label{subsec:Interpolation of a quotient map Borel ver}

This section corresponds to Section \ref{subsec:Interpolation of a quotient map}. See Definition \ref{def:Borel controlled coloring} for $c \colon \H_s(G) \to \{0,1\}$.

\begin{defn}\label{def:demormed quotient tildePsi}
    Given $\ell \in \NN$, define $w \colon \H_s(G)^2 \to [0,1]$ by
    \begin{align*}
        w(h,k) =
        \begin{cases}
        \frac{\ell}{\ell+1}    & {\rm if~} h = k\\
        \frac{1}{\ell+1}    & {\rm if~} (h,k) \in \H_s(G)^< \,\wedge\, c(k)=1 \\
        &~~~~\wedge\, \forall\,j \in \H_s(G),\, \big[(h,j) \in \H_s(G)^< \,\wedge\, (j,k) \in \H_s(G)^< \big] \Rightarrow c(j) = 0\\
        0 & {\rm otherwise}.
        \end{cases}
    \end{align*}
    Define $\K_c \subset \H_s(G)$ by $\K_c = c^{-1}(0)$ and let $s_{\K_c}$, $\H_{s_{\K_c}}(G_{\K_c})$, and $\Phi_{\K_c}^X$ be as in Proposition \ref{prop:E_K smooth}. Given $\xi \in C(G)$ (see Definition \ref{def:C(G)}), we define $\Psi_w\xi \in [0,1]^{\K_c}$ by
    \begin{align}\label{eq:tildePsi}
        (\Psi_w\xi)(k) = \min\big\{1, \sum_{h \in \H_s(G)}w(k,h)\xi(h)\big\}.
    \end{align}
    Define $\widetilde\Psi_w \colon C(G) \to C(G_{\K_c})$ by $\widetilde\Psi_w(a,\xi) = \big(\Phi_{\K_c}^X(a), \Psi_w\xi \circ (q_{\K_c}^X \times q_{\K_c}^X)|_{\K_c}^{-1}\big)$.
\end{defn}

\begin{rem}
    The sum $\sum_{h \in \H_s(G)}w(k,h)\xi(h)$ in \eqref{eq:tildePsi} is a finite sum by $\#\supp(\xi)<\infty$. See also Remark \ref{rem:Psi on C(H) well-defined}.
\end{rem}

\begin{lem}\label{lem:demormed quotient tildePsi is Borel}
The following hold.
\begin{itemize}
    \item[(1)]
    The map $w$ is Borel.
    \item[(2)]
    The map $\widetilde\Psi_w \colon C(G) \to C(G_{\K_c})$ is Borel.
\end{itemize}
\end{lem}

\begin{proof}
    (1) This follows since the sets $w^{-1}(\frac{\ell}{\ell+1})$, $w^{-1}(\frac{1}{\ell+1})$, and $w^{-1}(0)$ are Borel by Theorem \ref{thm:countable to one}, Lemma \ref{lem:Cr,Pa Borel}, and Lemma \ref{lem:c is Borel} (3).

    (2) Let $\mathcal{A} = (A_n)_{n\in\NN}$ and $\mathcal{B} = (B_n)_{n\in\NN}$ be Borel transversals for $E_G^{\H_s(G)}$ and $E_{G_{\K_c}}^{\H_{s_{\K_c}}(G_{\K_c})}$, respectively, such that $\H_s(G) = \bigcup_{n \in \NN}A_n$ and $\H_{s_{\K_c}}(G_{\K_c}) = \bigcup_{n \in \NN}B_n$ (see Remark \ref{rem:A for psi_A exists}). Let $\psi_\mathcal{A} \colon C(G) \to X/E_G^X \times [0,1]^\NN$ and $\psi_\mathcal{B} \colon C(G_{\K_c}) \to X_{\K_c}/E_{G_{\K_c}}^{X_{\K_c}} \times [0,1]^\NN$ be the maps defined in Definition \ref{def:C(G)} (see \eqref{eq:C(H_G)}). Define $Graph(\widetilde\Psi_w) \subset X/E_G^X \times [0,1]^\NN \times X_{\K_c}/E_{G_{\K_c}}^{X_{\K_c}} \times [0,1]^\NN$ by
    \begin{align*}
        Graph(\widetilde\Psi_w) = \{(\psi_\mathcal{A}(\eta), \psi_\mathcal{B}\circ\widetilde\Psi_w(\eta)) \mid \eta \in C(G)\}.
    \end{align*}
    For any $((a,\{b_n\}_{n \in \NN}),(a',\{b'_n\}_{n \in \NN})) \in X/E_G^X \times [0,1]^\NN \times X_{\K_c}/E_{G_{\K_c}}^{X_{\K_c}} \times [0,1]^\NN$, we have
    \begin{align*}
        &((a,\{b_n\}_{n \in \NN}),(a',\{b'_n\}_{n \in \NN})) \in Graph(\widetilde\Psi_w) \iff \\
        &\Phi_{\K_c}^X(a)=a' \,\wedge\,\Big[(\forall\, n\in\NN,\,b_n=b'_n=0) \,\vee\, \exists\, k \in \NN,\,\exists\, h_1,\cdots,h_k\in \H_s(G),\, \exists\,t_1,\cdots,t_k \in (0,1],\\
        &~~~~~~~\wedge_{i=1}^k\,q_G^{\H_s(G)}(h_i)=a \,\wedge\, \wedge_{i \neq j} \, h_i \neq h_j\\
        &~~~~~~~\wedge~\Big(\forall\,n \in\NN,\, \wedge_{i=1}^k\,(h_i \in A_n \Rightarrow \,b_n = t_i) \,\wedge\, \big((\forall\,i,\,h_i \notin A_n) \Rightarrow \,b_n = 0\big) \Big)\\
        &~~~~~~~\wedge~\Big(\forall\,n \in\NN,\, \exists\,h \in (q_{\K_c}^X \times q_{\K_c}^X)|_{\K_c}^{-1}(B_n),\, q_G^{\H_s(G)}(h) = a  \,\wedge\, b'_n = \min\big\{1, \sum_{i = 1}^k w(h,h_i)t_i\big\} \Big)\Big].
    \end{align*}
    Hence, $Graph(\widetilde\Psi_w)$ is Borel by countability of $E_G^X$, Theorem \ref{thm:countable to one}, and Lemma \ref{lem:demormed quotient tildePsi is Borel} (1), which implies that $\widetilde\Psi_w$ is Borel by Theorem \ref{thm:countable to one}.
\end{proof}

\subsection{The projection theorem}\label{subsec:The projection theorem Borel ver}

This section corresponds to Section \ref{subsec:The Wright's projection theorem}. The goal of this section is to prove Proposition \ref{prop:projection theorem Borel ver}. In Section \ref{subsec:Controlled colorings Borel ver} and Section \ref{subsec:Interpolation of a quotient map Borel ver}, we applied the Wright's construction to every $G$-component in the same way. However, in this section, we have to apply it differently depending on $G$-components so that it becomes Borel as a whole. More precisely, we have to choose nice total orders in Theorem \ref{thm:supp of p^op, p^<} (c) for each $G$-component to get a Borel projection.

\begin{defn}\label{def:proj thm Borel ver}
    Given $m \in \NNo$, we define $\H_s(G)_m^< \subset \H_s(G)^<$ by
    \begin{align*}
        \H_s(G)_m^< = \{(h,k) \in \H_s(G)^< \mid d_X(h^{(0)},k^{(0)}) = m\}.
    \end{align*}
    Define an equivalence relation $E_G^{\H_s(G)^<}$ on $\H_s(G)^<$ (resp. $E_G^{\H_s(G)_m^<}$ on $\H_s(G)_m^<$ and $E_G^{\H_s(G)^{op}}$ on $\H_s(G)^{op}$) as follows; $(h_1,k_1)$ and $(h_2,k_2)$ in $\H_s(G)^<$ (resp. $\H_s(G)_m^<$ and $\H_s(G)^{op}$) are equivalent if and only if $h_1 \,E_G^{\H_s(G)}\, h_2$. Given a Borel partial transversal $B \subset \H_s(G)^<$ for $E_G^{\H_s(G)^<}$, we define a map $p_B^< \colon C(G) \to C(G)$ by (see Definition \ref{def:p_h,k^<, p_h,k^op} for $p_{h,k}^<$)
    \begin{align*}
        p_B^<(a,\xi) =
        \begin{cases}
           (a, p_{h,k}^<(\xi)) & {\rm if~} \exists\,1\,(h,k) \in B, q_G^{\H_s(G)}(h)=a\\
           (a,\xi) & {\rm otherwise~}.
        \end{cases}
    \end{align*}
    Similarly, given a Borel partial transversal $B \subset \H_s(G)^{op}$ for $E_G^{\H_s(G)^{op}}$, we define a map $p_B^{op} \colon C(G) \to C(G)$ by (see Definition \ref{def:p_h,k^<, p_h,k^op} for $p_{h,k}^{op}$)
    \begin{align*}
        p_B^{op}(a,\xi) =
        \begin{cases}
           (a, p_{h,k}^{op}(\xi)) & {\rm if~} \exists\,1\,(h,k) \in B, q_G^{\H_s(G)}(h)=a\\
           (a,\xi) & {\rm otherwise~}.
        \end{cases}
    \end{align*}
    Given $a \in X/E_G^X$, we define $C_a = \{(x,\xi) \in C(G) \mid x=a \}$, $B_a = \{(x,\xi) \in \widetilde\iota(\widetilde X) \mid x=a \}$, $\H_a^{op} = \{(h,k) \in \H_s(G)^{op} \mid q_G^{\H_s(G)}(h) = a\}$, $\H_a^< = \{(h,k) \in \H_s(G)^< \mid q_G^{\H_s(G)}(h) = a\}$, and $\H_{a,m}^< = \{(h,k) \in \H_s(G)_m^< \mid q_G^{\H_s(G)}(h) = a\}$ where $m \in \NNo$.
\end{defn}

\begin{rem}
    Note $\H_s(G)^< = \bigsqcup_{m=0}^\infty \H_s(G)_m^<$. Also, for any $m \in \NNo$, every Borel partial transversal for $\H_s(G)_m^<$ is a Borel partial transversal for $\H_s(G)^<$.
\end{rem}

\begin{lem}\label{lem:proj thm Borel ver}
The following hold.
\begin{itemize}
    \item[(1)]
    $\H_s(G)_m^<$ is Borel for any $m \in \NNo$.
    \item[(2)] 
    $E_G^{\H_s(G)^<}$, $E_G^{\H_s(G)_m^<}$, and $E_G^{\H_s(G)^{op}}$ are a smooth CBER.
    \item[(2)] 
    For any Borel partial transversal $B$ for $E_G^{\H_s(G)^{op}}$ (resp. $E_G^{\H_s(G)^<}$), the map $p_B^{op}$ (resp. $p_B^<$) is Borel.
\end{itemize}
\end{lem}

\begin{proof}
    (1) For any $m \in \NNo$ and $(h,k) \in \H_s(G)^<$, we have by Remark \ref{rem:separate convex sets},
    \begin{align*}
        &(h,k) \in \bigsqcup_{i=m}^\infty \H_s(G)_i^< \\
        &\iff d_X(h^{(0)},k^{(0)}) \ge m \\
        &\iff \exists\,h_0,\cdots ,h_{m+1}\in \H_s(G),\, h_0=h \,\wedge\, h_{m+1}=k \,\wedge\, \wedge_{i=1}^{m+1}\, \{(h,h_i),(h_i,k)\}\subset \H_s(G)^<.
    \end{align*}
    Hence, $\bigsqcup_{i=m}^\infty \H_s(G)_i^<$ is Borel for any $m \in \NNo$ by countability of  $E_G^X$ and Theorem \ref{thm:countable to one}, hence so is $\H_s(G)_m^<$ by $\H_s(G)_m^<= \bigsqcup_{i=m}^\infty \H_s(G)_i^< \setminus \bigsqcup_{i=m+1}^\infty \H_s(G)_i^<$.
    
    (2) Since $E_G^X$ is a CBER, so is $E_G^{\H_s(G)^<}$. Define a map $f \colon \H_s(G)^< \to X/E_G^X$ by $f(h,k) = q_G^{\H_s(G)}(h)$, then $f$ is Borel and we have for any $(h_1,k_1), (h_2,k_2) \in \H_s(G)^<$, $(h_1,k_1) \,E_G^{\H_s(G)^<}\, (h_2,k_2) \iff f(h_1,k_1) = f(h_2,k_2)$. Thus, $E_G^{\H_s(G)^<}$ is smooth. The argument for $E_G^{\H_s(G)_m^<}$ and $E_G^{\H_s(G)^{op}}$ is the same.

    (3) Let $B \subset \H_s(G)^{op}$ be a Borel partial transversal for $E_G^{\H_s(G)^{op}}$. Let $\mathcal{A}=(A_n)_{n \in \NN}$ be Borel transversals for $E_G^{\H_s(G)}$ satisfying $\H_s(G) = \bigcup_{n \in \NN}A_n$ (see Remark \ref{rem:A for psi_A exists}) and let $\psi_\mathcal{A}$ be the map defined in Definition \ref{def:C(G)}. Define $Graph(p_B^{op})$ by
    \begin{align*}
        Graph(p_B^{op})
        =
        \{(\psi_\mathcal{A}(\eta), \psi_\mathcal{A}(p_B^{op}(\eta)) \in (X/E_G^X \times [0,1]^\NN)^2 \mid \eta \in C(G)\}.
    \end{align*}
    For any $((a,\{b_n\}_{n=1}^\infty),(a',\{b'_n\}_{n=1}^\infty)) \in (X/E_G^X \times [0,1]^\NN)^2$, we have
    \begin{align*}
        &((a,\{b_n\}_{n=1}^\infty),(a',\{b'_n\}_{n=1}^\infty)) \in Graph(p_B^{op})\\
        &\iff
         (a,\{b_n\}_{n=1}^\infty) \in \psi_\mathcal{A}(C(G)) \,\wedge\, a=a' \\
         &~~~~~~\wedge\,\bigg[\Big( \exists\,(h,k) \in B, \exists\, s,t \in \NN, q_G^{\H_s(G)}(h) = a \,\wedge\, h \in A_s \,\wedge\, k \in A_t\\
         &~~~~~~~~~~~~~~~~~~~~~~\wedge\, \forall\,n \in \NN,\Big[(h \in A_n \Rightarrow b'_n = p_1^{op}(b_s,b_t)) \,\wedge\, (k \in A_n \Rightarrow b'_n = p_2^{op}(b_s,b_t)) \\
         &~~~~~~~~~~~~~~~~~~~~~~~~~~~~~~~~~~~~~~~~~~~~~~~~~~\wedge\, \big((h \notin A_n \,\wedge\, k \notin A_n) \Rightarrow b'_n = b_n\big) \Big] \Big)\\
         &~~~~~~~~~~~~~\vee\Big(\big(\forall\, (h,k) \in B,\,q_G^{\H_s(G)}(h) \neq a\big) \wedge \big(\forall\, n \in \NN,\,b'_n=b_n\big) \Big)\bigg].
    \end{align*}
    Hence, $Graph(p_B^{op})$ is Borel by Theorem \ref{thm:countable to one}, which implies that $p_B^{op}$ is Borel by Theorem \ref{thm:countable to one}. The argument for $p_B^<$ is the same.
\end{proof}

Proposition \ref{prop:projection theorem Borel ver} is a Borel version of Wright's projection theorem (Theorem \ref{thm:supp of p^op, p^<}).

\begin{prop}\label{prop:projection theorem Borel ver}
    There exists a Borel map $P \colon C(G) \to \widetilde\iota(\widetilde X)\,(\cong \widetilde X)$ such that for every $a \in X/E_G^X$, $P|_{C_a} \colon C_a \to B_a$ is contractive and $P|_{B_a} = \mathrm{id}_{B_a}$.
\end{prop}

\begin{proof}
   By Lemma \ref{lem:seq of transversal} and Lemma \ref{lem:proj thm Borel ver} (2), there exist Borel partial transversals $(B_n)_{n \in \NN}$ for $E_G^{\H_s(G)^{op}}$ such that $\H_s(G)^{op} = \bigsqcup_{n \in \NN}B_n$. For every $(h,k) \in \H_s(G)^{op}$, there exists a unique $n_{h,k} \in \NN$ such that $(h,k) \in B_{n_{h,k}}$.
    
    Let $a \in X/E_G^X$. Define a total order $\prec_a$ on $\H_a^{op}$ as follows; for $(h_1,k_1), (h_2,k_2) \in \H_a^{op}$,
    \begin{align*}
    (h_1,k_1) \prec_a (h_2,k_2) \iff n_{h_1,k_1} < n_{h_2,k_2}.
    \end{align*}
    The family of maps $\{p_{B_n}^{op}|_{C_a}\}_{n \in (\NN,<)}$ (see Definition \ref{def:proj thm Borel ver}), where $p_{B_n}^{op}|_{C_a} \colon C_a \to C_a$, satisfies the conditions (1) and (2) of Lemma \ref{lem:Lemma 3.5 of Wri12} by applying Remark \ref{rem:H^op-supp(xi), H^<-supp(xi) finite} and Theorem \ref{thm:supp of p^op, p^<} (a) to $\H_a^{op}$. By Lemma \ref{lem:Lemma 3.5 of Wri12}, for any $\eta \in C_a$, there exists $N \in \NN$ such that for any $n \ge N$,
    \begin{align}\label{eq:p_{B_n}^{op}(eta)}
        p_{B_n}^{op}\circ \cdots \circ p_{B_1}^{op}(\eta) = P_{(\H_a^{op}, \prec_a)}(\eta).
    \end{align}
    
    For each $n \in \NN$, define $D_n \subset C(G)$ by
    \begin{align*}
        D_n = \{\eta \in C(G) \mid \forall\, m \ge n,\, p_{B_m}^{op}\circ \cdots \circ p_{B_1}^{op}(\eta) = p_{B_n}^{op}\circ \cdots \circ p_{B_1}^{op}(\eta)\}
    \end{align*}
    Since $\{p_{B_m}^{op}\}_{m \in \NN}$ are Borel by Lemma \ref{lem:proj thm Borel ver} (3), the set $D_n$ is Borel for any $n \in \NN$. By \eqref{eq:p_{B_n}^{op}(eta)}, for any $n \in \NN$ and $\eta = (a,\xi) \in D_n$, we have $P_{(\H_a^{op}, \prec_a)}(\eta) = p_{B_n}^{op}\circ \cdots \circ p_{B_1}^{op}(\eta)$. Hence, the map $P^{op} \colon C(G) \to C(G)$ defined by
    \begin{align*}
        P^{op}(a,\xi) = P_{(\H_a^{op}, \prec_a)}(a,\xi)
    \end{align*}
    is Borel on $D_n$ for any $n \in \NN$. Since we have $C(G) = \bigcup_{n \in \NN}D_n$ by \eqref{eq:p_{B_n}^{op}(eta)}, $P^{op}$ is Borel.
    
    Similarly, by Lemma \ref{lem:seq of transversal} and Lemma \ref{lem:proj thm Borel ver} (2), for every $m \in \NNo$, there exist Borel partial transversals $(B_{m,n})_{n \in \NN}$ for $E_G^{\H_s(G)_m^<}$ such that $\H_s(G)_m^< = \bigsqcup_{n \in \NN}B_{m,n}$. Note $\H_s(G)^< = \bigsqcup_{m \in \NNo} \H_s(G)_m^< = \bigsqcup_{(m,n) \in (\NNo)\times \NN}B_{m,n}$. Hence, for every $(h,k) \in \H_s(G)^<$, there exists a unique $(m_{h,k}, n_{h,k}) \in (\NNo) \times \NN$ such that $(h,k) \in B_{m_{h,k}, n_{h,k}}$.

    Let $a \in X/E_G^X$. Define a total order $\prec_a$ on $\H_a^<$ as follows; for $(h_1,k_1), (h_2,k_2) \in \H_a^<$,
    \begin{align*}
    (h_1,k_1) \prec_a (h_2,k_2) \iff m_{h_1,k_1} > m_{h_2,k_2} \,\vee\, (m_{h_1,k_1} = m_{h_2,k_2} \,\wedge\, n_{h_1,k_1} < n_{h_2,k_2}).
    \end{align*}

    For $m \in \NNo$, define $E_m \subset C(G)$ by
    \begin{align*}
        E_m = \{\eta \in C(G) \mid \forall\,m' > m, \forall\, n \in\NN,\, p^<_{B_{m',n}}(\eta) = \eta\}.
    \end{align*}
    Since $\{p_{B_{m,n}}^<\}_{m \in \NN}$ are Borel maps by Lemma \ref{lem:proj thm Borel ver} (3), the set $E_m$ is Borel for any $m \in \NN$. Define $E_{a,m} \subset C_a$ by $E_{a,m} = \{(x,\xi) \in E_m \mid x=a\}$. For any $\eta \in E_{a,m}$ and $(h,k) \in \H_a^<\hy\supp(\eta)$, we have $d_X(h^{(0)}, k^{(0)}) \le m$. Hence, by Theorem \ref{thm:supp of p^op, p^<} (b), we can see that for any $(h,k) \in \H_{a,m}^<$, we have $p_{h,k}^<(E_{a,m}) \subset E_{a,m}$. Also, the triple $E_{a,m}$, $(\H_{a,m}^<, \prec_a)$, and $\{p_{h,k}^<\}_{(h,k) \in \H_{a,m}^<}$ satisfy the conditions (1) and (2) of Lemma \ref{lem:Lemma 3.5 of Wri12}. By Lemma \ref{lem:Lemma 3.5 of Wri12}, for any $\eta  \in E_{a,m}$, we have $\H_{a,m}^<\hy\supp(P_{(\H_{a,m}^<, \prec_a)}(\eta)) = \emptyset$ and there exists $N \in \NN$ such that for any $n \ge N$,
    \begin{align}\label{eq:p_{B_{m,n}}^<}
        p_{B_{m,n}}^<\circ \cdots \circ p_{B_{m,1}}^<(\eta) = P_{(\H_{a,m}^<, \prec_a)}(\eta)
    \end{align}
    By $\H_{a,m}^<\hy\supp(P_{(\H_{a,m}^<, \prec_a)}(\eta)) = \emptyset$, we have $P_{(\H_{a,m}^<, \prec_a)}(E_{a,m}) \subset E_{a,m-1}$ for any $m \in \NN$. Define a map $P_m^< \colon E_m \to E_{m-1}$ by $P_m^<(a,\xi) = P_{(\H_{a,m}^<, \prec_a)}(a,\xi)$. By \eqref{eq:p_{B_{m,n}}^<}, we can see that $P^<_m$ is Borel in the same way as $P^{op}$. Hence, the map $P_0^<\circ \cdots \circ P_m^< \colon E_m \to E_0$ is Borel for any $m \in \NNo$. Also, for any $m \in \NNo$ and $\eta =(a,\xi) \in E_m$, we have
    \begin{align*}
        P_0^<\circ \cdots \circ P_m^< (\eta) = P_{(\H_a^<, \prec_a)} (\eta).
    \end{align*}
    Thus, the map $P^< \colon C(G) \to C(G)$ defined by $P^<(a,\xi) = P_{(\H_a^<, \prec_a)} (a, \xi)$ is Borel on $E_m$ for any $m \in \NNo$. Since we have $C(G) = \bigcup_{m \in \NNo} E_m$ by Remark \ref{rem:H^op-supp(xi), H^<-supp(xi) finite}, the map $P^<$ is Borel.

    Thus, the map $P \colon C(G) \to C(G)$ defined by $P = P^< \circ P^{op}$ is Borel. By Theorem \ref{thm:supp of p^op, p^<} (c), we have $P(C(G)) \subset \widetilde\iota(\widetilde X)$ and $P|_{C_a} \colon C_a \to B_a$ is contractive and $P|_{B_a} = \mathrm{id}_{B_a}$ for any $a \in X/E_G^X$.
\end{proof}

\section{Proof of Theorem \ref{thm:main}}\label{sec:Proof of the main theorem}

In this section, we prove Proposition \ref{prop:intro smooth median graph} (Proposition \ref{prop:smooth median graph}) and then Theorem \ref{thm:main} (Theorem \ref{thm:main again}). For this, we first prove Lemma \ref{lem:Lipschitz map implies finite Bas}, which is a Borel version and also a variant of a characterization of asymptotic dimension using a map to a simplicial complex.

From Definition \ref{def:triangulation of [-1,1]^n} up to Lemma \ref{lem:stars}, we explain how to triangulate a Borel graph ($\widetilde X$ more precisely) in a Borel way, which will be applied to prove Lemma \ref{lem:Lipschitz map implies finite Bas}. The point is that we triangulate preserving symmetry of cubes to ensure Borel measurability (see Remark \ref{rem:S_ell,n G invariant}).

\begin{defn}\label{def:triangulation of [-1,1]^n}
 Let $n \in \NN$. We triangulate $[-1,1]^n (\subset \RR^n)$ by cutting by the set of $(n-1)$-dimensional hyperplanes $H_1$ in $\RR^n$ defined by $x_i = 0$ for each $i\in \{1,\cdots,n\}$ and $H_2$ defined by $x_i = x_j$ and $x_i = -x_j$ for each distinct $i,j \in \{1,\cdots,n\}$. 
\end{defn}

\begin{rem}
    We can see that this is a triangulation as follows. The set $H_1$ of hyperplanes cuts $[-1,1]^n$ into $2^n$ cubes of side length 1. The cube $[0,1]^n$ is one of them. The set $H_2$ of hyperplanes cuts $[0,1]^n$ into $n$-dimensional simplexes that are defined by
    \begin{align*}
        \{(x_1,\cdots,x_n) \in [0,1]^n \mid x_{\tau(1)} \le \cdots \le x_{\tau(n)}\}
    \end{align*}
    for each permutation $\tau \in Sym(n)$ of coordinates. By symmetry, the other cubes of side length 1 are triangulated by $H_2$ in the same way as $[0,1]^n$.
\end{rem}

The idea of taking the second barycentric subdivision below is from \cite[Theorem 9.9 (d) $\Rightarrow$ (a)]{Roe03}.

\begin{defn}\label{def:triangulation T_2}
 Define a map $f \colon \RR^n \to \RR^n$ by $f(x_1,\cdots,x_n) = \big(\frac{1}{2}(x_1+1),\cdots,\frac{1}{2}(x_n+1)\big)$. By $f([-1,1]^n) = [0,1]^n$, we define the triangulation of $[0,1]^n$ by mapping the triangulation of $[-1,1]^n$ in Definition \ref{def:triangulation of [-1,1]^n} by $f$. Let $T \subset 2^{[0,1]^n}$ be the set of all simplexes in this triangulation of $[0,1]^n$. Here, the number of $n$-simplexes in $T$ is $2^n \cdot \#Sym(n)$. Let $T_1 \subset 2^{[0,1]^n}$ be the set of all simplexes obtained by the barycentric subdivision of $T$. Similarly, let $T_2 \subset 2^{[0,1]^n}$ be the set of all simplexes obtained by the barycentric subdivision of $T_1$. Each simplex in $T_2$ is closed in $[0,1]^n$. For each $0$-simplex $v \in T_1$, define $\st_{T_2}(v) \subset [0,1]^n$ by
 \begin{align*}
     \st_{T_2}(v) = \bigcup\{\sigma \in T_2 \mid v \in \sigma\}.
 \end{align*}
 For each $\ell \in \{0,\cdots,n\}$, let $V_{\ell,n} \subset T_1$ be the set of all $0$-simplexes in $T_1$ that are the barycenter of some $\ell$-simplex in $T$. Define $S_{\ell,n} \subset [0,1]^n$ by 
 \begin{align*}
     S_{\ell,n} = \bigcup_{v \in V_{\ell,n}}\st_{T_2}(v).
 \end{align*}
\end{defn}

\begin{rem}
    We identify simplexes as their geometric realization in $[0,1]^n$.
\end{rem}

\begin{rem}\label{rem:S_ell,n G invariant}
    The set $S_{\ell,n}$ is closed, hence Borel, in $[0,1]^n$. Also, $S_{\ell,n}$ is $\S_n$-invariant.
\end{rem}

\begin{figure}[htbp]

\begin{minipage}[c]{0.5\hsize}
\begin{center}
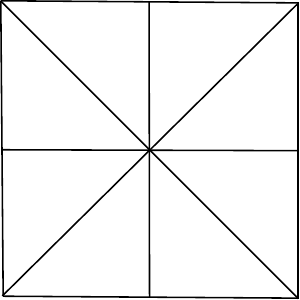
 \caption{Triangulation $T$ of $[0,1]^2$} 
 \label{Fig:T}
\end{center}
\end{minipage} 
\hfill %Don't put empty line here!
\begin{minipage}[c]{0.5\hsize}
\begin{center}
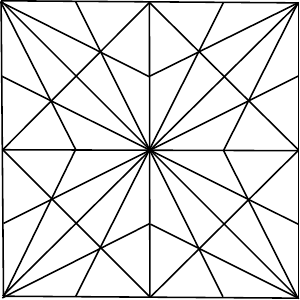
 \caption{Triangulation $T_1$ of $[0,1]^2$}
 \label{Fig:T1}
\end{center}
\end{minipage} 
\end{figure}

\begin{defn}\label{def:triangulate}
Let $X$ be a standard Borel space and $G \subset X^2$ be a countable Borel median graph. We define the triangulation of each $G$-component of $\widetilde X$ by triangulating each $n$-cube in the same way as $T$ for $[0,1]^n$ in Definition \ref{def:triangulation T_2}. Let $\T \subset 2^{\widetilde X}$ be the set of all simplexes of this triangulation of $\widetilde X$. Let $\T_1 \subset 2^{\widetilde X}$ be the set of all simplexes obtained by the barycentric subdivision of $\T$ and let $\T_2 \subset 2^{\widetilde X}$ be the set of all simplexes obtained by the barycentric subdivision of $\T_1$. For each $0$-simplex $v \in \T_1$, define $\st_{\T_2}(v) \subset \widetilde X$ by
\begin{align*}
    \st_{\T_2}(v) = \bigcup\{\sigma \in \T_2 \mid v \in \sigma\}.
\end{align*}
For each $\ell \in \NNo$, let $\V_\ell \subset \T_1$ be the set of all $0$-simplexes in $\T_1$ that are the barycenter of some $\ell$-simplex in $\T$. Define $\SS_\ell \subset \widetilde X$ by
\begin{align*}
    \SS_\ell = \bigcup_{v \in \V_\ell}\st_{\T_2}(v).
\end{align*}
\end{defn}

\begin{figure}[htbp]
    \centering
    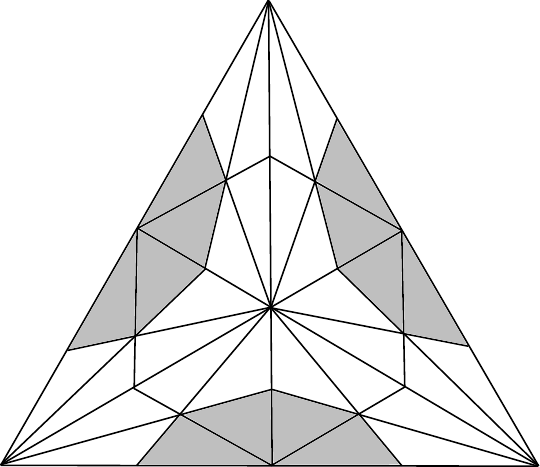
    \caption{The shaded region is the intersection of $\SS_1$ and a 2-simplex in $\T$}
    \label{Fig:subdivision}
\end{figure}

Lemma \ref{lem:stars} (2) below is a known fact because the argument reduces to each connected component, which is a triangulation of a CAT(0) cube complex endowed with the $\ell^1$-metric, but we will write down a sketch of the proof for completeness.

\begin{lem}\label{lem:stars}
Let $X$ be a standard Borel space and $G \subset X^2$ be a countable Borel median graph. Then, the following hold.
\begin{itemize}
    \item[(1)] 
     The set $\SS_\ell$ is Borel for any $\ell \in \NNo$.
    \item[(2)]
     If there exists $D \in \NNo$ such that every $G$-component is a CAT(0) cube complex of dimension at most $D$, then there exists $\delta>0$, depending only on $D$, such that for any $\ell \in \NNo$ and any $v,w \in \V_\ell$ with $v \neq w$, we have $d_{\widetilde X}(\st_{\T_2}(v), \st_{\T_2}(w))\ge \delta$.
\end{itemize}

\end{lem}

\begin{proof}
    (1) As in Definition \ref{def:sigma-algebra on widetilde X}, fix Borel transversals $\mathcal{P} = (P_0)_{n \in \NNo}$ for $E^{C_n}_{\S_n}$ for each $n \in \NNo$ and let $F_\mathcal{P} \colon \bigsqcup_{n \in \NNo}P_n \times (0,1)^n \to \widetilde X$ be the map defined by \eqref{eq:F_mathcal P}. For any $n \in \NNo$, we have $\big(P_n \times (0,1)^n\big) \cap F_\mathcal{P}^{-1}(\SS_\ell) = P_n \times \big((0,1)^n \cap S_{\ell,m}\big)$ since $S_{\ell,n}$ is $\S_n$-invariant (see Remark \ref{rem:S_ell,n G invariant}). Hence, $\big(P_n \times (0,1)^n\big)\cap F_\mathcal{P}^{-1}(\SS_\ell)$ is Borel. Thus, $\SS_\ell$ is Borel.
    
    (2) Let $\ell \in \NNo$. For $v \in \V_\ell$, define $\st_{\T_1}(v) \subset \widetilde X$ by $\st_{\T_1}(v) = \bigcup\{\sigma \in \T_1 \mid v \in \sigma\}$. For each $i =1,2$, we denote by $\partial \st_{\T_i}(v)$ the boundary of $\st_{\T_i}(v)$.
    
    We claim that there exists $\delta > 0$, depending only on $D$, such that
    \begin{equation}\label{eq:distant from outer sphere}
        \inf_{v \in \V_\ell}d_X(\st_{\T_2}(v), \partial\st_{\T_1}(v)) \ge \delta.
    \end{equation}
    Let $v \in \V_\ell$, $x \in \st_{\T_2}(v)$ and $y \in \partial\st_{\T_1}(v)$. Since every $G$-component is a CAT(0) cube complex and line segments in a cube with its $\ell^1$-metric can be parametrized to be geodesic, we can take a geodesic $\gamma$ from $x$ to $y$ and a sequence of points $x=x_0,\cdots,x_n=y$ on $\gamma$ such that for every $i>0$, the subpath $\gamma_{[x_{i-1}, x_i]}$ of $\gamma$ from $x_{i-1}$ to $x_i$ is a line segment contained in some simplex in $\T_1$. There exist $i,j>0$ with $i \le j$ and points $x' \in \gamma_{[x_{i-1}, x_i]}$ and $y' \in \gamma_{[x_{j-1}, x_j]}$ such that $x' \in \partial\st_{\T_2}(v)$, $y' \in \partial\st_{\T_1}(v)$, and $\gamma_{[x',y']}\setminus\{x'\} \subset \st_{\T_1}(v) \setminus \st_{\T_2}(v)$. 
    
    If $i+2 \le j$, then there exists $\delta_1 > 0$, depending only on $D$, such that $|\gamma_{[x_i, x_{i+1}]}| \ge \delta_1$ since $\gamma_{[x_i, x_{i+1}]}$ is contained in some simplex in $\T_1$ but not in $\st_{\T_2}(v)$. Hence, $d_{\widetilde X}(x,y) = |\gamma| \ge \delta_1$. 
    
    If $i \le j \le i+1$, then there exists $\delta_2 > 0$, depending only on $D$, such that $|\gamma_{[x', y']}| \ge \delta_2$ since $\gamma_{[x', y']}$ goes from $x' \in \partial\st_{\T_2}(v)$ to $y' \in \partial\st_{\T_1}(v)$ within at most 2 simplexes in $\T_1$. Thus, \eqref{eq:distant from outer sphere} follows with $\delta = \min\{\delta_1, 
    \delta_2\}$.

    Let $v,w \in \V_\ell$ with $v \neq w$. Assume that $v$ and $w$ are in the same $G$-component since otherwise we just have $d_{\widetilde X}(\st_{\T_2}(v), \st_{\T_2}(w)) = \infty$. For any $x \in \st_{\T_2}(v)$ and $y \in \st_{\T_2}(w)$, every geodesic from $x$ to $y$ passes through $\partial\st_{\T_1}(v)$ by $y \notin \st_{\T_1}(v)$. This implies $d_{\widetilde X}(x,y) \ge \delta$ by \eqref{eq:distant from outer sphere}. Hence, the statement follows with $\delta$. 
\end{proof}

Now, we are ready to prove Lemma \ref{lem:Lipschitz map implies finite Bas}. We prepare terminologies to state Lemma \ref{lem:Lipschitz map implies finite Bas}.

\begin{defn}
    Let $X$ and $Y$ be extended metric spaces. We call a map $f \colon Y \to X$
    \begin{itemize}
        \item[-]
        \emph{$\e$-Lipschitz}, where $\e \in (0,\infty)$, if $\forall\, x,y \in Y,\,d(f(x),f(y))\le \e \cdot d(x,y)$.
        \item[-]
        \emph{cobornologous} if for any $r \in (0,\infty)$, there exists $R \in (0,\infty)$ such that $$\forall\, x,y \in Y,\,d(f(x),f(y))\le r \Rightarrow d(x,y) \le R.$$
    \end{itemize}
\end{defn}

\begin{lem}\label{lem:Lipschitz map implies finite Bas}
    Let $D \in \NNo$ and let $(Y,\rho)$ be a Borel extended metric space such that $E_\rho$ is countable. Suppose that for any $\e \in (0,\infty)$, there exist a standard Borel space $X$, a countable Borel median graph $G \subset X^2$, and a $\e$-Lipschitz cobornologous Borel map $f \colon Y \to \widetilde X$ such that every $G$-component is a CAT(0) cube complex of dimension at most $D$. Then, $\mathrm{asdim}_{\mathbf{B}}(Y,\rho) \le D$.
\end{lem}

\begin{proof}
    Let $r > 0$. Let $\delta>0$ be the constant as in Lemma \ref{lem:stars} (2) for $D$. Take $X$, $G$, and $f$ as in the condition for $\e = (r+1)^{-1}\delta$. For $(X,G)$, define $\T$, $\T_1$, $\T_2$, $\st_{\T_2}(v)$, $\V_\ell$, and $\SS_\ell$ as in Definition \ref{def:triangulate}. By Lemma \ref{lem:stars} (1), $\SS_\ell$ is Borel. Define $U_\ell \subset Y$ by $U_\ell = f^{-1}(\SS_\ell)$, then $U_\ell$ is Borel since $f$ is Borel. We have $Y = \bigcup_{\ell = 0}^D U_\ell$ by $\widetilde X = \bigcup_{\ell = 0}^D \SS_\ell$. Note $U_\ell = \bigsqcup_{v \in \V_\ell} f^{-1}(\st_{\T_2}(v))$. For any $\ell \ge 0$ and any $v,w \in \V_\ell$ with $v \neq w$, we have
    \begin{align*}
        \rho(f^{-1}(\st_{\T_2}(v)), f^{-1}(\st_{\T_2}(w))\ge \e^{-1}\delta \ge r+1.
    \end{align*}
    by Lemma \ref{lem:stars} (2) since $f$ is $\e$-Lipschitz. Hence, every $\mathfrak{F}_r(U_\ell)$-component is contained in $f^{-1}(\st_{\T_2}(v))$ for some $v \in \V_\ell$. Also, we have $\sup_{v \in \V_\ell}\diam_Y(f^{-1}(\st_{\T_2}(v)))<\infty$ since we have $\sup_{v \in \V_\ell}\diam_{\widetilde X}(\st_{\T_2}(v)) < \infty$ and the map $f$ is cobornologous ($\diam_Y, \diam_{\widetilde X}$ denote the diameter in $Y, \widetilde X$). Hence, $\mathfrak{F}_r(U_\ell)$ is uniformly bounded for any $\ell \in \{0,\cdots,D\}$. Thus, $\mathrm{asdim}_{\mathbf{B}}(Y,\rho) \le D$.
\end{proof}

The idea of the proof of Proposition \ref{prop:intro smooth median graph} below is that the Wright's construction of a Lipschitz cobornologous has some flexibility, so we can find a nice Wright's construction for each component of a Borel graph so that they become a Borel map as a whole.

\begin{prop}\label{prop:smooth median graph}{\rm (Proposition \ref{prop:intro smooth median graph})}
    Let $X$ be a standard Borel space and $G \subset X^2$ be a countable Borel median graph such that $E_G^X$ is smooth. Suppose that there exists $D \in \NNo$ such that every $G$-component is a CAT(0) cube complex of dimension at most $D$. Then, $\mathrm{asdim}_{\mathbf{B}}(X,G) \le D$.
\end{prop}

\begin{proof}
    Let $\e >0$. Let $s\colon X\to X$ be a Borel selector for $E_G^X$ and $\H_s(G)$ be a space of hyperplanes of $G$ for $s$. By Theorem \ref{thm:coloring} and Lemma \ref{lem:c is Borel} (3), there exists a Borel map $c \colon \H_s(G) \to \{0,1\}$ such that the restriction of $c$ to every $G$-component is a $3^{D-1}D$-controlled coloring.
    
    Set $\K_c=c^{-1}(0)$ and let $(X_{\K_c},G_{\K_c})$ be the quotient of $(X,G)$ by $\K_c$ (see Definition \ref{def:notions on H_s(G)}). By Proposition \ref{prop:E_K smooth}, $X_{\K_c}$ is a standard Borel space, $G_{\K_c}$ is a countable Borel median graph, and $E_{G_{\K_c}}^{X_{\K_c}}$ is smooth. Also, every $G_{\K_c}$-component is a countable CAT(0) cube complex of dimension at most $D$.
    
    Define the map $\widetilde\Psi_w \colon C(G) \to C(G_{\K_c})$ as in Definition \ref{def:demormed quotient tildePsi}. By Lemma \ref{lem:demormed quotient tildePsi is Borel}, $\widetilde\Psi_w$ is Borel. By \cite[Lemma 4.8]{Wri12}, the restriction of $\widetilde\Psi_w|_{\widetilde X} \colon \widetilde X \to C(G_{\K_c})$ to every $G$-component is $\frac{3^{D-1}D}{3^{D-1}D+1}$-Lipschitz (see Remark \ref{rem:widetilde X identified}). (By \cite[Lemma 4.7]{Wri12}, the quotient map $\widetilde X \to \widetilde{X_{\K_c}}$ is cobornologous on $X$.)
    
    By Proposition \ref{prop:projection theorem Borel ver}, there exists a Borel contractive map $P \colon C(G_{\K_c}) \to \widetilde{X_{\K_c}}$ such that $P|_{\widetilde{X_{\K_c}}}=\mathrm{id}_{\widetilde{X_{\K_c}}}$ (note Remark \ref{rem:iota is isom}). By arguing in the same way as the proof of \cite[Theorem 4.9]{Wri12} for each $G$-component, we can see that $P \circ \widetilde\Psi_w|_{\widetilde X} \colon \widetilde X \to \widetilde{X_{\K_c}}$ is a Borel map whose restriction to every $G$-component is a $\frac{3^{D-1}D}{3^{D-1}D+1}$-Lipschitz cobornologous map to its quotient. 
    
    Take $N \in \NN$ satisfying $\big(\frac{3^{D-1}D}{3^{D-1}D+1}\big)^N < \e$. By repeatedly taking a quotient in the above way $N$ times, there exist a standard Borel space $Y$, a countable Borel median graph $G'$ on $Y$, and a $\e$-Lipschitz cobornologous Borel map $f \colon X \to \widetilde Y$ such that $E_{G'}^Y$ is smooth and every $G'$-component is a CAT(0) cube complex of dimension at most $D$. By Lemma \ref{lem:Lipschitz map implies finite Bas}, $\mathrm{asdim}_{\mathbf{B}}(X,G) \le D$.
\end{proof}

See Section \ref{sec:CAT(0) cube complex} and Section \ref{sec:Descriptive set theory} for relevant notations.

\begin{thm}\label{thm:main again}{\rm (Theorem \ref{thm:main})}
    For any countable CAT(0) cube complex $X$ of dimension $D \in \NNo$, $G_{\R(X)}$ is a countable Borel graph on $\R(X)$, $E_{G_{\R(X)}}^{\R(X)}$ is smooth, and we have $\mathrm{asdim}_{\mathbf{B}}(\R(X),G_{\R(X)}) \le D$.
\end{thm}

\begin{proof}
    Since $X$ is countable, $\R(X)$ is a compact metric space, hence Polish. Since the set $\H(X)$ of hyperplanes is countable, every $G_{\R(X)}$-component is a countable median graph (see Remark \ref{rem:Roller boundary}) and the corresponding CAT(0) cube complex has dimension at most $D$. For any $x,y \in \R(X)$, we have
    \begin{align*}
        (x,y)\in G_{\R(X)} \iff \exists\,h\in \H(X),\,x(h)\neq y(h) \,\wedge\,\big[\forall\, k \in \H(X)\setminus\{h\},\,x(k)=y(k)\big].
    \end{align*}
    Hence, $G_{\R(X)}$ is Borel. By Proposition \ref{prop:smooth median graph}, it is enough to prove that $E_{G_{\R(X)}}^{\R(X)}$ is smooth in what follows. We set $E = E_{G_{\R(X)}}^{\R(X)}$ for brevity.

    Fix $o \in X^{(0)}$. For any $h \in \H(X)$, fix the choice of orientation for $h$ by $o \in h^-$ and define $\overline{h^-}, \overline{h^+} \subset \R(X)$ by $\overline{h^-} = \{x \in \R(X) \mid x(h) = h^-\}$ and $\overline{h^+} = \{x \in \R(X) \mid x(h) = h^+\}$. By \cite[Proposition 2.9]{Gen21}, for any $x \in \R(X)$, there exists a unique point $\xi_x \in [x]_E$ such that the hyperplanes separating $o$ and $\xi_x$ are precisely the hyperplanes separating $o$ and $[x]_E$. Define a map $f \colon \R(X) \to \R(X)$ by $f(x) = \xi_x$. By uniqueness of $\xi_x$, the map $f$ satisfies $x \,E_{G_{\R(X)}}^{\R(X)}\, y \iff f(x) = f(y)$ for any $x,y \in \R(X)$. 
    
    Let $x \in \R(X)$. Since the hyperplanes separating $o$ and $\xi_x$ are precisely the hyperplanes separating $o$ and $[x]_E$, we have for any $h \in \H(X)$,
    \begin{align}\label{eq:smooth 1}
        \xi_x(h) =
        \begin{cases}
        h^+ 
        & {\rm if~} [x]_E \subset \overline{h^+},\\
        \alpha_o(h)=h^- 
        &{\rm if~} [x]_E \cap \overline{h^+}\neq\emptyset {\rm ~and~} [x]_E \cap \overline{h^-} \neq\emptyset,\\
        h^-
        & {\rm if~}[x]_E\subset \overline{h^-}.
        \end{cases}
    \end{align}
    This implies $f^{-1}(\overline{h^-}) = \{x\in \R(X) \mid [x]_E \cap \overline{h^-} \neq\emptyset\} = \big[\overline{h^-}\big]_E$. Hence, $f^{-1}(\overline{h^-})$ is Borel for any $h \in \H(X)$ by countability of $E$ and Theorem \ref{thm:countable to one}. Hence, $f$ is Borel, which implies that $E_{G_{\R(X)}}^{\R(X)}$ is smooth.
\end{proof}

%% bibliographystyle should be either amsplain or amsalpha

%\bibliography{main.bib}
%\bibliographystyle{amsalpha}

\input{main.bbl}

\vspace{5mm}

\noindent  Department of Mathematics, Vanderbilt University, Nashville 37240, U.S.A.

\noindent E-mail: \emph{koichi.oyakawa@vanderbilt.edu}

\end{document}

%% file: ttriangulationsvg.pdf_tex
%% Creator: Inkscape 1.2 (dc2aedaf03, 2022-05-15), www.inkscape.org
%% Accompanies image file 'ttriangulationsvg.pdf' (pdf, eps, ps)
\begingroup%
  \makeatletter%
  \providecommand\color[2][]{%
    \errmessage{(Inkscape) Color is used for the text in Inkscape, but the package 'color.sty' is not loaded}%
    \renewcommand\color[2][]{}%
  }%
  \providecommand\transparent[1]{%
    \errmessage{(Inkscape) Transparency is used (non-zero) for the text in Inkscape, but the package 'transparent.sty' is not loaded}%
    \renewcommand\transparent[1]{}%
  }%
  \providecommand\rotatebox[2]{#2}%
  \newcommand*\fsize{\dimexpr\f@size pt\relax}%
  \newcommand*\lineheight[1]{\fontsize{\fsize}{#1\fsize}\selectfont}%
  \ifx\svgwidth\undefined%
    \setlength{\unitlength}{143.54905148bp}%
    \ifx\svgscale\undefined%
      \relax%
    \else%
      \setlength{\unitlength}{\unitlength * \real{\svgscale}}%
    \fi%
  \else%
    \setlength{\unitlength}{\svgwidth}%
  \fi%
  \global\let\svgwidth\undefined%
  \global\let\svgscale\undefined%
  \makeatother%
  \begin{picture}(1,1.00074934)%
    \lineheight{1}%
    \setlength\tabcolsep{0pt}%
    \put(0,0){\includegraphics[width=\unitlength,page=1]{ttriangulationsvg.pdf}}%
  \end{picture}%
\endgroup%

%% file: t1triangulationsvg.pdf_tex
%% Creator: Inkscape 1.2 (dc2aedaf03, 2022-05-15), www.inkscape.org
%% Accompanies image file 't1triangulationsvg.pdf' (pdf, eps, ps)
\begingroup%
  \makeatletter%
  \providecommand\color[2][]{%
    \errmessage{(Inkscape) Color is used for the text in Inkscape, but the package 'color.sty' is not loaded}%
    \renewcommand\color[2][]{}%
  }%
  \providecommand\transparent[1]{%
    \errmessage{(Inkscape) Transparency is used (non-zero) for the text in Inkscape, but the package 'transparent.sty' is not loaded}%
    \renewcommand\transparent[1]{}%
  }%
  \providecommand\rotatebox[2]{#2}%
  \newcommand*\fsize{\dimexpr\f@size pt\relax}%
  \newcommand*\lineheight[1]{\fontsize{\fsize}{#1\fsize}\selectfont}%
  \ifx\svgwidth\undefined%
    \setlength{\unitlength}{143.71102342bp}%
    \ifx\svgscale\undefined%
      \relax%
    \else%
      \setlength{\unitlength}{\unitlength * \real{\svgscale}}%
    \fi%
  \else%
    \setlength{\unitlength}{\svgwidth}%
  \fi%
  \global\let\svgwidth\undefined%
  \global\let\svgscale\undefined%
  \makeatother%
  \begin{picture}(1,0.99962142)%
    \lineheight{1}%
    \setlength\tabcolsep{0pt}%
    \put(0,0){\includegraphics[width=\unitlength,page=1]{t1triangulationsvg.pdf}}%
  \end{picture}%
\endgroup%

%% file: subdivision.pdf_tex
%% Creator: Inkscape 1.2 (dc2aedaf03, 2022-05-15), www.inkscape.org
%% Accompanies image file 'subdivision.pdf' (pdf, eps, ps)
\begingroup%
  \makeatletter%
  \providecommand\color[2][]{%
    \errmessage{(Inkscape) Color is used for the text in Inkscape, but the package 'color.sty' is not loaded}%
    \renewcommand\color[2][]{}%
  }%
  \providecommand\transparent[1]{%
    \errmessage{(Inkscape) Transparency is used (non-zero) for the text in Inkscape, but the package 'transparent.sty' is not loaded}%
    \renewcommand\transparent[1]{}%
  }%
  \providecommand\rotatebox[2]{#2}%
  \newcommand*\fsize{\dimexpr\f@size pt\relax}%
  \newcommand*\lineheight[1]{\fontsize{\fsize}{#1\fsize}\selectfont}%
  \ifx\svgwidth\undefined%
    \setlength{\unitlength}{258.92795364bp}%
    \ifx\svgscale\undefined%
      \relax%
    \else%
      \setlength{\unitlength}{\unitlength * \real{\svgscale}}%
    \fi%
  \else%
    \setlength{\unitlength}{\svgwidth}%
  \fi%
  \global\let\svgwidth\undefined%
  \global\let\svgscale\undefined%
  \makeatother%
  \begin{picture}(1,0.86567283)%
    \lineheight{1}%
    \setlength\tabcolsep{0pt}%
    \put(0,0){\includegraphics[width=\unitlength,page=1]{subdivision.pdf}}%
  \end{picture}%
\endgroup%

%% file: main.bbl
\newcommand{\etalchar}[1]{$^{#1}$}
\providecommand{\bysame}{\leavevmode\hbox to3em{\hrulefill}\thinspace}
\providecommand{\MR}{\relax\ifhmode\unskip\space\fi MR }
\providecommand{\MRhref}[2]{%
  \href{http://www.ams.org/mathscinet-getitem?mr=#1}{#2}
}
\providecommand{\href}[2]{#2}

%% file: main.bbl
\begin{thebibliography}{CJM{\etalchar{+}}23}

\bibitem[BY25]{BY24}
Anton Bernshteyn and Jing Yu, \emph{Large-scale geometry of {B}orel graphs of polynomial growth}, Adv. Math. \textbf{473} (2025), Paper No. 110290. \MR{4897446}

\bibitem[Che00]{Che00}
Victor Chepoi, \emph{Graphs of some {${\rm CAT}(0)$} complexes}, Adv. in Appl. Math. \textbf{24} (2000), no.~2, 125--179. \MR{1748966}

\bibitem[CJM{\etalchar{+}}23]{CJM23}
Clinton~T. Conley, Steve~C. Jackson, Andrew~S. Marks, Brandon~M. Seward, and Robin~D. Tucker-Drob, \emph{Borel asymptotic dimension and hyperfinite equivalence relations}, Duke Math. J. \textbf{172} (2023), no.~16, 3175--3226. \MR{4679959}

\bibitem[CPTT25]{CPTT23}
Ruiyuan Chen, Antoine Poulin, Ran Tao, and Anush Tserunyan, \emph{Tree-like graphings, wallings, and median graphings of equivalence relations}, Forum Math. Sigma \textbf{13} (2025), Paper No. e64, 37. \MR{4883098}

\bibitem[DSU17]{DSU17}
Tushar Das, David Simmons, and Mariusz Urba\'nski, \emph{Geometry and dynamics in {G}romov hyperbolic metric spaces}, Mathematical Surveys and Monographs, vol. 218, American Mathematical Society, Providence, RI, 2017, With an emphasis on non-proper settings. \MR{3558533}

\bibitem[Gao09]{Gao09}
Su~Gao, \emph{Invariant descriptive set theory}, Pure and Applied Mathematics (Boca Raton), vol. 293, CRC Press, Boca Raton, FL, 2009. \MR{2455198}

\bibitem[Gen21]{Gen21}
Anthony Genevois, \emph{Rank-one isometries of {CAT}(0) cube complexes and their centralisers}, Adv. Geom. \textbf{21} (2021), no.~3, 347--364. \MR{4283595}

\bibitem[GH24]{GH24}
Jan Grebík and Cecelia Higgins, \emph{Complexity of finite borel asymptotic dimension}, 2024.

\bibitem[IS24]{IS24}
Sumun Iyer and Forte Shinko, \emph{Asymptotic dimension and hyperfiniteness of generic cantor actions}, 2024.

\bibitem[Kec95]{Kec95}
Alexander~S. Kechris, \emph{Classical descriptive set theory}, Graduate Texts in Mathematics, vol. 156, Springer-Verlag, New York, 1995. \MR{1321597}

\bibitem[Kec25]{Kec25}
\bysame, \emph{The theory of countable {B}orel equivalence relations}, Cambridge Tracts in Mathematics, vol. 234, Cambridge University Press, Cambridge, 2025. \MR{4837611}

\bibitem[NV25]{NV23}
Petr Naryshkin and Andrea Vaccaro, \emph{Hyperfiniteness and {B}orel asymptotic dimension of boundary actions of hyperbolic groups}, Math. Ann. \textbf{392} (2025), no.~1, 197--208. \MR{4887757}

\bibitem[Roe03]{Roe03}
John Roe, \emph{Lectures on coarse geometry}, University Lecture Series, vol.~31, American Mathematical Society, Providence, RI, 2003. \MR{2007488}

\bibitem[Rol98]{Rol98}
Martin Roller, \emph{Poc sets, median algebras and group actions}, 1998, habilitation, Universität Regensburg.

\bibitem[Sag95]{Sag95}
Michah Sageev, \emph{Ends of group pairs and non-positively curved cube complexes}, Proc. London Math. Soc. (3) \textbf{71} (1995), no.~3, 585--617. \MR{1347406}

\bibitem[Sag14]{Sag14}
\bysame, \emph{{$\rm CAT(0)$} cube complexes and groups}, Geometric group theory, IAS/Park City Math. Ser., vol.~21, Amer. Math. Soc., Providence, RI, 2014, pp.~7--54. \MR{3329724}

\bibitem[Tse22]{Anu22}
Anush Tserunyan, \emph{Introduction to descriptive set theory}, 2022, Lecture note, \url{https://www.math.mcgill.ca/atserunyan/Teaching_notes/dst_lectures.pdf}.

\bibitem[Wei24]{Wei24}
Felix Weilacher, \emph{Borel edge colorings for finite-dimensional groups}, Israel J. Math. \textbf{263} (2024), no.~2, 737--780. \MR{4819965}

\bibitem[Wis12]{Wis12}
Daniel~T. Wise, \emph{From riches to raags: 3-manifolds, right-angled {A}rtin groups, and cubical geometry}, CBMS Regional Conference Series in Mathematics, vol. 117, Conference Board of the Mathematical Sciences, Washington, DC; by the American Mathematical Society, Providence, RI, 2012. \MR{2986461}

\bibitem[Wis21]{Wis21}
\bysame, \emph{The structure of groups with a quasiconvex hierarchy}, Annals of Mathematics Studies, vol. 209, Princeton University Press, Princeton, NJ, [2021] \copyright 2021. \MR{4298722}

\bibitem[Wri12]{Wri12}
Nick Wright, \emph{Finite asymptotic dimension for {${\rm CAT}(0)$} cube complexes}, Geom. Topol. \textbf{16} (2012), no.~1, 527--554. \MR{2916293}

\end{thebibliography}
